\documentclass[11pt,a4paper]{article}
\usepackage[margin=1in]{geometry}
\usepackage{multicol}

\usepackage{authblk}
\usepackage{amssymb,latexsym, bm, dsfont,graphicx,subfigure}
\usepackage{amsfonts, amsmath}
\usepackage{enumerate}
\usepackage{algorithm}
\usepackage{stmaryrd}
\usepackage{float}
\usepackage{stmaryrd}
\usepackage{hyperref}
\usepackage[noend]{algpseudocode}
\usepackage{subfigure,appendix}
\usepackage{amsthm}
\usepackage{thmtools}
\usepackage{thm-restate}
\usepackage{tabularx}
\usepackage{nicefrac}

\newtheorem{theorem}{Theorem}[section]

\newtheorem{remark}[theorem]{Remark}
\newtheorem{lemma}[theorem]{Lemma}

\newtheorem{claim}[theorem]{Claim}

\usepackage{url}
\urldef{\mailsa}\path|dadush@cwi.nl,|
\urldef{\mailsb}\path|{l.vegh, g.zambelli}@lse.ac.uk,|

\newcommand{\bb}{\mathbb}
\newcommand{\conv}{\mathrm{conv}}

\newcommand{\R}{\bb R}

\newcommand{\Z}{\bb Z}

\newcommand{\B}{\bb B}

\newcommand{\ceil}[1]{\lceil#1\rceil}

\newcommand{\intr}{\mathrm{int}}

\newcommand{\sm}{\setminus}

\newcommand{\scalar}[1]{\left\langle #1\right\rangle}

\newcommand{\T}{\mathsf{T}}
\newcommand{\F}{{\cal F}}

\newcommand{\bk}{\ensuremath{b}}

\newcommand{\Down}[1]{{#1}^{\downarrow}}
\newcommand{\Up}[1]{{#1}^{\uparrow}}

\newcommand{\low}{\ensuremath{\ell}}
\newcommand{\lf}{\ensuremath{f^\downarrow}}

\newcommand{\EO}{\mathrm{EO}}
\newcommand{\AO}{\mathrm{AO}}

\def\tr{\mathrm{tr}}

\def\st{\,:\,}

\newcommand{\Fm}[1]{\ensuremath{F_{#1}}}
\newcommand{\Lb}[1]{\ensuremath{L_{#1}}}
\newcommand{\Lbt}[1]{\ensuremath{L_{#1,2}}}

\newcommand{\eqdef}{\mathbin{\stackrel{\rm def}{=}}}

\def\k{\kappa}

\floatstyle{plain}
\newfloat{myalgo}{tbhp}{mya}

\newcommand\ourpagewidth{.9\textwidth}

{\begin{myalgo}[#1]
\centering
\begin{minipage}{\ourpagewidth}
\begin{algorithm}[H]}%
{\end{algorithm}
\end{minipage}
\end{myalgo}}

\newif\ifnotes
\notesfalse

\ifnotes
\usepackage{color}
\definecolor{mygrey}{gray}{0.50}
\newcommand{\notename}[2]{{\textcolor{mygrey}{\footnotesize{\bf (#1:} {#2}{\bf ) }}}}

\newcommand{\gnote}[1]{{\notename{Giacomo}{#1}}}
\newcommand{\lnote}[1]{{\notename{Laci}{#1}}}
\newcommand{\dnote}[1]{{\notename{Daniel}{#1}}}

\else

\newcommand{\notename}[2]{{}}

\newcommand{\gnote}[1]{}
\newcommand{\lnote}[1]{}
\newcommand{\dnote}[1]{}

\fi

\begin{document}

\title{Geometric Rescaling Algorithms for Submodular Function
  Minimization\thanks{An extended abstract of this paper was presented at SODA 2018.}}
\author[1]{
Daniel Dadush\thanks{Supported by NWO Veni grant 639.071.510 and ERC Starting Grant QIP--805241.}}
\author[2]{L\'aszl\'o A. V\'egh\thanks{Supported by EPSRC First Grant EP/M02797X/1, and by  ERC
  Starting Grant ScaleOpt--757481.}}
  \author[2]{Giacomo Zambelli}
   \affil[ ]{ \texttt{dadush@cwi.nl}, \texttt{\{l.vegh,g.zambelli\}@lse.ac.uk}}
  \affil[1]{Centrum Wiskunde \& Informatica, Amsterdam, The Netherlands}
  \affil[2]{London School of Economics and Political Science, London, UK}


\date{}

\maketitle

\abstract{
We present a new class of polynomial-time algorithms for submodular function minimization (SFM), as well as a unified framework to obtain strongly polynomial SFM algorithms. Our algorithms are based on simple iterative methods for the minimum-norm problem, such as the conditional gradient and Fujishige-Wolfe algorithms. We exhibit two techniques to turn simple iterative methods into polynomial-time algorithms.

Firstly, we adapt the geometric rescaling technique, which has recently gained attention in linear programming, to SFM and obtain a weakly polynomial bound
$O(({n}^4\cdot \EO + {n}^5)\log ({n} L))$.

Secondly, we exhibit a general combinatorial black-box approach to turn  $\varepsilon L$-approximate SFM oracles into strongly polynomial exact SFM algorithms. This  framework can be applied to a wide range of combinatorial and continuous algorithms, including pseudo-polynomial ones. In particular, we can obtain strongly polynomial  algorithms by a repeated application of the conditional gradient or of the Fujishige-Wolfe algorithm. Combined with the geometric rescaling technique, the black-box approach provides an $O(({n}^5\cdot \EO +{n}^6)\log^2{n})$ algorithm.

Finally, we show that one of the techniques we develop in the paper can also be combined with the cutting-plane method of Lee, Sidford, and Wong \cite{LSW}, yielding a simplified variant of their $O(n^3 \log^2 n \cdot \EO + n^4\log^{O(1)} n)$ algorithm. }

\section{Introduction}

Given a finite set $V$, a function $f: 2^V\to \mathbb{R}$ is {\em submodular}  if
\begin{equation}\label{eq:submod}
f(X)+f(Y)\geq f(X\cap Y)+f(X\cup Y)\quad \forall X,Y\subseteq V.
\end{equation}
We denote $n:=|V|$. Examples include the graph cut function, the
coverage function, and the entropy function. Submodularity can be
interpreted as a diminishing returns property and is therefore
important in economics and game theory. Submodular optimization is
widely applied in machine learning and computer vision (see e.g. \cite{Bach-monograph}).

We will assume that the function $f$ is given via an {\em
  evaluation oracle}: for every set $S\subseteq V$, we can query the value
$f(S)$ in time $\EO$.
We will assume throughout that $f(\emptyset)=0$; this is without loss of
generality.
 In the {\em submodular function minimization (SFM)} problem,
the objective is to find a minimizer of this function:
\begin{equation}\label{prob:sfm}
\min_{S\subseteq V} f(S).\tag{SFM}
\end{equation}
The first weakly polynomial-time algorithm was given by Gr\"otschel,
Lov\'asz, and Schrijver in 1981 \cite{Grotschel1981}, whereas the first strongly polynomial algorithm was given by the
same authors in their 1988 book \cite{glsbook}. Both algorithms used the ellipsoid method. It remained an important goal to find a
strongly polynomial combinatorial algorithm; this question was resolved independently by Schrijver \cite{schrijver2000} and by Iwata,
Fleischer, and Fujishige \cite{IFF}  in
2000. Currently the best running time of a combinatorial
algorithm is $O(n^5\cdot\EO+n^6)$ by Orlin \cite{orlin2009}. In a recent
breakthrough, Lee, Sidford, and Wong \cite{LSW} provided a  $O(n^3 \log^2 n \cdot \EO + n^4
\log^{O(1)} n)$ algorithm based on a new cutting-planes method.

However, the above algorithms do not appear to work well for large
scale instances that arise in applications such as speech recognition or image
segmentation. A recent line of work has focused on exploiting special
structures of specific submodular functions arising in these applications,
such as decomposability \cite{EneN15,Ene17,jegelka11,StobbeK10}, but for general functions
simple iterative algorithms appear to outperform the provably
polynomial algorithms \cite{fujishige11}. In particular, the Fujishige-Wolfe
minimum-norm point
algorithm \cite{fujishige80,wolfe}  appears to be among the best ones
in practice \cite{Bach-monograph,fujishige11}, despite the
fact that  the first pseudo-polynomial running time bound was
given as recently as 2014 by Chakrabarty et al. \cite{chakrabarty14}.

\paragraph{Our contributions} This paper presents polynomial-time algorithms based on simple
iterative methods such as the conditional gradient algorithm or the
Fujishige-Wolfe algorithm. We exhibit two different techniques to
improve the performance of these algorithms to polynomially
bounded. The first technique uses \emph{geometric rescaling}, whereas
the second provides a \emph{unified combinatorial framework for strongly polynomial SFM
algorithms}. In what follows, we provide an overview of both techniques.

\medskip\noindent
\emph{Geometric rescaling} has recently gained attention in the context of
linear programming. This is a general algorithmic technique to turn simple iterative algorithms to polynomial-time
  algorithms for LP feasibility, by adaptively changing the scalar
  product. The first such algorithms were given  by Betke \cite{betke}, and by Dunagan and Vempala
  \cite{Dunagan-Vempala}, and a number of papers have since appear on the subject. We refer the reader to \cite{part-1} for an overview of the
  literature.

In this paper we focus on one such algorithm, introduced by the authors in \cite{part-1} and named there the Full Support Image Algorithm (the same algorithm was obtained independently by Hoberg and Rothvo\ss{} \cite{rothvoss}). We show how this algorithm can be adapted to \eqref{prob:sfm}. In Section~\ref{sec:remarks} we will see that the framework we introduce is robust, in
the sense that it can be easily adapted to different rescaling
algorithms so long as they are applicable to conic problems in the
separation oracle model, such as the algorithms in
\cite{belloni,chubanov-oracle,part-1,Pena-Soheili} (see for example a
recent note of Fujishige~\cite{Fujishige-oracle} showing how an
algorithm of Chubanov~\cite{chubanov-oracle} can be used in this
framework). The reason for focusing on the Image Algorithm of
\cite{part-1} is that, within this framework, it provides the best
running time bounds for  \eqref{prob:sfm} among the known algorithms
in this class.

We introduce new techniques that enable our Image Algorithm
to provide both primal and dual optimal solutions for \eqref{prob:sfm}.
 The \emph{sliding technique} is used to obtain a primal
optimal solution: we reduce the optimization problem \eqref{prob:sfm}
to a dynamically changing feasibility problem.
The \emph{pull-back
  technique} enables to identify a dual optimality certificate. Moreover, the same technique  allows us also to
  obtain approximate dual solutions, and it is also applicable in the general LP feasibility setting.

For integer valued submodular functions, our geometric rescaling algorithm finds both primal and dual optimal solutions, in
  running time $O((n^4 \cdot\EO + n^5)\log (n L))$, where the
  complexity parameter $L$ denotes the largest norm of a point in the
  base polytope. This matches the best weakly polynomial guarantees
  \cite{iwata-faster,Iwata-Orlin} prior to the work of Lee, Sidford, and Wong \cite{LSW}.

\medskip

\noindent\emph{Unified combinatorial
  framework.} Building on the geometric rescaling technique, we also obtain a \emph{strongly polynomial} $O((n^5 \cdot\EO +
  n^6)\log^2 n)$ algorithm. This is obtained from a unified combinatorial
  framework which allows us to turn any algorithm that can produce
a $\delta L$-approximate solution to \eqref{prob:sfm}  in pseudo-polynomial
poly$(n,1/\delta)$ running time, into an exact strongly
polynomial algorithm. More specifically, if we are given an oracle that can produce a set $W\subseteq V$ and a point $y$ in the base polytope such that
$f(W)\leq y^-(V)+\delta L_f$ (where $y^-(V)$ is the sum of all negative components of $y$ and $L_f$ is the largest $1$-norm of any point in the base polytope), then with the choice $\delta\in \Theta(1/n^3)$  we can find an optimal solution to  \eqref{prob:sfm} in roughly $O(n^2)$ calls to the oracle.

Hence, somewhat
surprisingly, even pseudo-polynomial time algorithms such as the conditional gradient
or the Fujishige-Wolfe algorithm immediately give rise to strongly polynomial  time algorithms.

We emphasise that our approach is \emph{black-box}: we explicitly formulate the approximate oracle requirement, and show that any routine fulfilling such requirements
provides a strongly polynomial-time algorithm for \eqref{prob:sfm}. To illustrate our point, we will
show that various routines in the literatures, such as \cite{IFF,Iwata-Orlin,LSW}, satisfy such black-box requirements. We believe that our approach is more modular than the previous works, in the sense that in all previous papers the combinatorial arguments on strongly polynomial progress were intertwined with the details of some ``basic'' routine.

We can also apply this unified framework to the cutting plane
method. Using the general cutting plane algorithm of Lee, Sidford, and Wong \cite{LSW}, we show that
our black-box method can recover their running time bound $O(n^3 \log^2 n\cdot \EO + n^4
\log^{O(1)} n)$ for \eqref{prob:sfm}.
This is made possible by the use of the same sliding technique developed for our geometric rescaling algorithm.

The general combinatorial framework is based on maintaining a \emph{ring family} guaranteed to
contain all  minimizer sets, where the size of the family decreases through the algorithm until a minimizer is found. This technique was introduced by Iwata, Fleischer, and Fujishige \cite{IFF}, and used in multiple subsequent
  papers, such as Iwata and Orlin \cite{Iwata-Orlin}, and Lee, Sidford,
  and Wong \cite{LSW}. We note that this technique ultimately traces back to
  strongly polynomial algorithms for minimum-cost flows, pioneered by
  Tardos \cite{Tardos85}. Our implementation also adopts a simplified variant of the bucketing
  technique of \cite{LSW} that leads to a
  factor $n$ improvement in the running time compared to the
  original framework of \cite{IFF}.

\medskip

The rest of the paper is structured as
follows. Section~\ref{sec:prelim} contains definitions and the
necessary background, including an overview of the relevant iterative
methods. Section~\ref{sec:weakly} presents the weakly polynomial
geometric rescaling algorithm to solve SFM, while in Subsection~\ref{sec:farkas}, we describe the pull-back technique that
enables the implementation of the approximate oracle using our
geometric rescaling method. Section~\ref{sec:strongly} presents the
general framework for strongly polynomial algorithms. Section~\ref{sec:ellipsoid} shows how
cutting plane methods, and in particular the cutting plane algorithm in \cite{LSW}, can be used in the strongly polynomial framework.
Finally, in Section~\ref{sec:remarks} we discuss variants of the weakly-polynomial geometric rescaling algorithm.


\section{Preliminaries}\label{sec:prelim}
We refer the reader to \cite{fujishige-book} and \cite[Sections 44-45]{Schrijver03} for the
basics of submodular optimization; these references contain all definitions as well as
the proof of the results presented next. The survey \cite{Bach-monograph}
provides an overview of continuous algorithms for submodular function
minimization.

For a vector $z\in\mathbb{R}^V$, we denote by $z(v)$ the component of $z$ relative to $v\in V$, and for  a subset
$S\subseteq V$ we use the notation $z(S)=\sum_{v\in S}z(v)$.
We let $\|z\|_1:=\sum_{v\in V}|z(v)|$ and $\|z\|_2:=(\sum_{v\in
  V}|z(v)|^2)^{1/2}$ denote the $\ell_1$ and $\ell_2$-norms,
respectively; we will also use the simpler
$\|z\|:=\|z\|_2$.
For a
number $a\in \mathbb{R}$, we let $a^+=\max\{0,a\}$ and
$a^-=\min\{0,a\}$; hence, $a=a^++a^-$. Similarly, given a vector $z\in\mathbb{R}^V$, we denote $z^+=(z(v)^+)_{v\in V}$ and $z^-=(z(v)^-)_{v\in V}$.

We denote by $\bb{S}^{{n}}_{++}$ the set of ${n}\times{n}$ symmetric positive definite real matrices. Every $Q\in\bb{S}^{{n}}_{++}$ defines
the scalar product $\scalar{x,y}_Q\eqdef x^\T Qy$, inducing the norm
$\|x\|_Q\eqdef\sqrt{\scalar{x,x}_Q}$.

\paragraph{The base polytope and the greedy algorithm}
Let $f$ be a submodular function, where we assume as usual that $f(\emptyset)=0$. The  {\em base polytope} of $f$ is defined as
\begin{align*}
B(f) := \left\{ {x \in \mathbb{R}^V \colon x(S) \leq f(S) \; \forall S \subseteq V,\; x(V) = f(V) }\right\}.
\end{align*}
This polytope $B(f)$ is non-empty for every submodular function $f$. Its elements are called {\em bases}, and its vertices are the
{\em extreme bases}.
Extreme bases correspond to permutations of the
ground set. More precisely, for any ordering $v_1,v_2,\ldots,v_{n}$ of the elements of $V$, the following point is a vertex of $B(f)$, and every vertex is of this form for some ordering:
\begin{equation}\label{eq:vertex}
\begin{aligned}
x(v_1)&:=f(\{v_1\}),\\
x(v_i)&:=f(\{v_1,\ldots,v_i\})-f(\{v_1,\dots,v_{i-1}\})\quad \forall
i=2,\ldots, {n}.
\end{aligned}
\end{equation}
Furthermore, given a weight function  $w:V\to \R$, one can compute an
extreme base minimizing $w^\T x$ by the greedy algorithm
\textsc{GreedyMin}$(f,w)$ as follows: order the vertices in $V$ so
that $w(v_1)\le w(v_2)\le\ldots\le w(v_{n})$, and output $x$ defined by \eqref{eq:vertex} as the optimal solution.
The value of the minimum-cost is then given by
\begin{equation}\label{eq:vertex-cost}
\min_{x\in B(f)} w^\T x=\sum_{i=1}^{n-1} f(\{v_1,\ldots,v_i\})(w(v_i)-w(v_{i+1}))+f(V)w(v_n).
\end{equation}

The subroutine \textsc{GreedyMin}$(f,w)$ requires $O({n}\cdot\EO+{n}\log {n})$
arithmetic operations. If $w$ has several entries with the same value, then there
are multiple ways to sort the elements of $V$ in ascending value of $w$, each ordering potentially giving rise to a different optimal extreme base
of $B(f)$.  The extreme bases corresponding to the possible
tie-breakings are the vertices of the face of $B(f)$ minimizing
$w^\T x$.

If $v_1,\ldots,v_n$ is the ordering computed by \textsc{GreedyMin}$(f,w)$, we define
\begin{equation}\label{eq:minset}
\textsc{MinSet}(f,w)\eqdef\mbox{argmin}\{ f(S)\st S=\{v_1,\ldots,v_i\} \exists i\in [n] \mbox{ or } S=\emptyset\}.
\end{equation}

A min-max characterization of \eqref{prob:sfm} was given by Edmonds:
\begin{theorem}[Edmonds \cite{edmonds1970}]\label{thm:edmonds}
For any submodular function $f:2^V\to \mathbb{R}$,
\begin{equation}
\max\{x^-(V)\st x\in B(f)\}=\min\{f(S)\st S\subseteq V\}.\label{eq:edmonds}
\end{equation}
\end{theorem}
We will often use the following simple consequence. Assume that for some $x\in B(f)$, $S\subseteq V$, and $\varepsilon>0$,
we have $f(S)\le
x^-(V)+\varepsilon$. Then $f(S)\le f(T)+\varepsilon$ for any
$T\subseteq V$.

\paragraph{Complexity parameters}
When dealing with weakly polynomial time algorithms for \eqref{prob:sfm}, various complexity parameters have been considered in the literature to
measure the running time. All these parameters turn out to be
equivalent within an $O(n)$ factor. Nonetheless, in certain parts of the paper
different choices will fit more easily, hence we introduce all of them below. We define
\[\Lb{f}\eqdef \max\{\|z\|_1:z\in B(f)\},\ \ \Lbt{f}\eqdef
\max\{\|z\|_2:z\in B(f)\},\ \ \Fm{f}\eqdef \max\{|f(S)|:
S\subseteq V\}.\]

Some of our algorithms require the explicit knowledge of these complexity parameters. While these are hard to compute, we introduce easily computable upper-bounds that are essentially equivalent (see Claim~\ref{claim:complexity parameters}) to the ones above. These upper-bounds are given by the norms of the vector $\alpha:V\to\R$ defined by
\[\alpha(v):=\max\{|f(\{v\})|,|f(V)-f(V\setminus\{v\})|\},\quad v\in V.\]

In the following claim, we highlight that all these parameters are essentially equivalent  within an $O(n)$ factor.

\begin{claim}\label{claim:complexity parameters} For any submodular function $f\,:\, 2^V\to \R$ such that $f(\emptyset)=0$, the following hold.
\begin{enumerate}[(i)]
 \item $\Lbt{f}\le\Lb{f}\le \sqrt{n}\Lbt{f}$.
  \item $\Lb{f}=\Theta(\Fm{f})$.
  \item For every $z\in B(f)$, $|z(v)|\le \alpha(v)$ for all $v\in V$. Hence,  $\Lb{f}\le \|\alpha\|_1=\alpha(V)$ and $\Lbt{f}\le \|\alpha\|_2$.
  \item $\alpha(v)\le 2\Fm{f}$ for all $v\in V$.
  \end{enumerate}
\end{claim}
\begin{proof} (i) is obvious, for (ii) see e.g. \cite[Lemma 5]{chakrabarty16}, and also \cite{hazan2012,jegelka-bilmes2011}, for (iii) see \cite[Section 3.3]{fujishige-book}, and  (iv) follows immediately from the definition of $\alpha$.
\end{proof}

Our running time bounds will contain terms of the form $\log (nZ)$ for various choices of $Z\in
\{\Fm{f},\Lb{f},\Lbt{f},\|\alpha\|_1,\|\alpha\|_2\}$. The above lemma shows that all these terms are within a constant factor of one another,
hence the specific choice of complexity parameter does not matter.

\paragraph{The minimum-norm point problem}
Fujishige \cite{fujishige80} showed a reduction of \eqref{prob:sfm} to
the following convex
quadratic optimization problem.

\begin{theorem}[Fujishige \cite{fujishige80}]\label{thm:fuji}
Let $z$ be the unique optimal solution to
\begin{equation}
\min\left\{ \frac12 \|x\|^2_2\colon x \in B(f)\right\}.\label{eq:min-norm}
\end{equation}
Then, the set $S^*=\{v \in V\colon z(v) < 0\}$ is the inclusionwise minimal minimizer of \eqref{prob:sfm}, and $f(S^*)=z^-(V)$. Furthermore,  $|f(S^*)|\le \sqrt{n} \|z\|_2$.
\end{theorem}


Note that in case of $f(V)=0$, Theorems~\ref{thm:edmonds} and
~\ref{thm:fuji} imply that the
minimizer of the 2-norm also minimizes the 1-norm  in $B(f)$. Indeed, in this case for any $y\in B(f)$ $y^-(V)=f(V)-y^+(V)=-y^+(V)$, implying $\|y\|_1=-2 y^-(V)$, which gives
$\|z\|_1=-2z^-(V)=-2f(S^*)=-2\max\{x^-(V)\st x\in B(f)\}=\min\{\|x\|_1 \st x\in B(f)\}$.

Note that $z\in B(f)$ is the minimum norm point if and only if $z^\top x\ge \|z\|_2^2$ for all $x\in B(f)$. An approximate optimal solution
to \eqref{eq:min-norm} can be converted to an approximate optimal
solution to \eqref{eq:edmonds}, as stated below.
\begin{theorem}[Bach
\cite{Bach-monograph}]
\label{thm:approx-convert}
Assume that $z\in B(f)$ satisfies $\|z\|^2_2\le z^\T
x+\varepsilon$ for all $x\in B(f)$, for some $\varepsilon>0$. Let
$S=\textsc{MinSet}(f,z)$. Then, $f(S)\le z^-(V)+\sqrt{2n\varepsilon}$.
Consequently, $f(S)\le f(T)+\sqrt{2n\varepsilon}$
 for any
$T\subseteq V$.
\end{theorem}
The above is a consequence of the more general statement in \cite[Proposition 10.5]{Bach-monograph}. Since it might not be immediate to see such implication, for completeness we provide a proof of the previous theorem in the appendix.

\subsection{Iterative methods for SFM}\label{sec:iterative methods}
Convex optimization algorithms can be naturally applied to SFM, either by solving the quadratic
formulation~\eqref{eq:min-norm}, or by minimizing the so-called {\em
Lov\'asz-extension}, which we do not discuss here. We refer the reader to \cite{Bach-monograph} for
an overview of such algorithms. Here, we briefly outline two important algorithms
based on \eqref{eq:min-norm}.

\paragraph{The conditional gradient algorithm}
The conditional gradient, or Frank-Wolfe algorithm, maintains a point
$y\in B(f)$, represented as a convex combination $y=\sum_{i=1}^k \lambda_ig_i$
of extreme bases. It is initialized with $y=g$ for an arbitrary
extreme base $g$. Every iteration runs \textsc{GreedyMin}$(f,y)$ to
obtain  an extreme base $g'$. If $y^\T g'\ge \|y\|_2^2$, then $y$ is
the minimum-norm point in $B(f)$, and the algorithm
terminates. Otherwise, $y$ is replaced by the minimum-norm point
$y'$ on the line segment  $[y,g']$. This is also known as the von Neumann algorithm, described by Dantzig \cite{Dantzig-92}. The algorithm was first used  to test membership
in the base polytope by Sohoni~\cite{Sohoni}.
The standard convergence analysis of the conditional gradient algorithm and Theorem~\ref{thm:approx-convert}
provide the following convergence bound (see
e.g. \cite[Sec 10.8]{Bach-monograph}).

\begin{theorem}\label{thm:cond-grad}
For any $\delta>0$, within $O(n/\delta^2)$ iterations the
conditional gradient algorithm computes $y\in B(f)$ such that for
$S=\textsc{MinSet}(f,y)$, we have
  $f(S)\le y^-(V)+ O(\delta\Lbt{f})$. The total running time is
  $O((n^2 \cdot\EO+n^2\log n)/\delta^2)$.
\end{theorem}

\paragraph{The Fujishige-Wolfe algorithm} Wolfe \cite{wolfe} gave
a finite algorithm for finding the minimum-norm point in
a polytope given by its vertices; his algorithm can also be
interpreted as an active set method \cite{Bach-monograph}.
Fujishige adapted  Wolfe's algorithm to SFM
\cite{fujishige80,fujishige11}.
We now give a brief sketch of the
algorithm; for a more detailed description see
\cite{chakrabarty14,fujishige11,wolfe}.

An affinely independent set of points  $X\in \mathbb{R}^n$ is called a
\emph{corral} if the orthogonal projection of $0$ onto the affine hull of $X$
is in the relative interior of the convex hull of $X$. In particular, the optimal
solution to the minimum-norm point problem can be obtained by a
corral, comprising vertices of the minimal face of the
polytope containing the minimum-norm point.

Every \emph{major cycle} of the Fujishige-Wolfe algorithm starts and
ends with a
corral formed by extreme bases in $B(f)$. The algorithm is initialized
with an arbitrary extreme base  (note that every singleton set is a corral). Let $X$ be the
corral at the beginning of a major cycle, and let $y$ be the
projection of $0$ onto the affine hull of $X$; this can be obtained by a
closed-form formula. Let us run \textsc{GreedyMin}$(f,y)$ to obtain an
extreme base $g'$ minimizing $y^\T g'$.  If $y^\T g'\ge \|y\|_2^2$, then the algorithm
terminates returning $y$ as
the minimum-norm point in $B(f)$. Otherwise, we consider
$X'=X\cup\{g'\}$, which is also affinely independent. We set $\bar x=
y$, and compute $y'$ as
the projection of $0$ onto the affine hull of $X'$.  If $y'$ is in the relative interior of
$\conv(X')$, the major cycle terminates with the new corral
$X'$. Otherwise, we start a \emph{minor cycle}: we replace $X'$ by the
extreme points of the minimal face of  $\conv(X')$  intersecting the line segment
$[\bar x,y']$; the new
$\bar x$ is defined to be the unique intersection point.
Minor cycles are repeated until a corral is obtained. Finite
convergence is guaranteed since
$\|\bar x\|_2$ decreases in every major and minor cycle, and the
number of corrals is finite.
However, a bound on the convergence rate, which we summarize below, was only recently given in
\cite{chakrabarty14}. Further improvements on the analysis were observed by Lacoste and Jaggi in \cite{Lacoste-Julien}.

\begin{theorem}[Chakrabarty et al. \cite{chakrabarty14}, Lacoste and Jaggi \cite{Lacoste-Julien}]\label{thm:wolfe}
For any $\delta>0$, within $O(n/\delta^2)$ iterations (major and
minor cycles)  Fujishige-Wolfe algorithm computes $y\in B(f)$ such that,
for
$S=\textsc{MinSet}(f,y)$, we have
  $f(S)\le y^-(V)+ O(\delta\Lbt{f})$. The total running time is $O((n^2\cdot \EO+n^3)/\delta^2)$.
\end{theorem}


\section{Weakly polynomial algorithm via rescaling}\label{sec:weakly}

Throughout this section we assume that $f$ is an {\bf integer valued} submodular function. The assumption that $f$ is integer
valued is needed in the context of weakly polynomial-time algorithms..

\paragraph{The geometric rescaling algorithm} The Full Support Image Algorithm in \cite[Section 3.2]{part-1} is applicable to the
following oracle setting. Let $\Sigma\subseteq \mathbb{R}^n$ be
non-empty, full
dimensional cone; our aim is to find a feasible point in the interior.
We are given a separation oracle for $\intr(\Sigma)$; that is, for any
vector $w$, the oracle decides whether $w\in \intr(\Sigma)$, and if
not, it returns a vector $z$ such that $z^\T w\le 0$ but $z^\T
y> 0$ for all $y\in \intr(\Sigma)$. Then the algorithm finds a point in
$\intr(\Sigma)$ in  $O(n^3 \log \omega^{-1})$ calls to the separation oracle, where
$\omega$ is a condition number which we will define in
Section~\ref{sec:weakly-analysis}. We remark that the parameter $\omega$ can be
lower bounded by the {\em width of the cone} $\Sigma$, defined as the radius of the largest ball contained in $\Sigma$ and centered on the surface of the unit sphere.
\medskip

Consider now a submodular function $f$ with $f(V)=0$.
Assume we want to decide whether $f(S)\ge 0$ for all $S\subseteq V$,
that is, if $S=\emptyset$ is an optimal solution to
\eqref{prob:sfm}. It follows from the definition of the base polytope that
$0\leq f(S)$ for every $S\subseteq V$ if and only if  $0\in B(f)$ (note that
$f(V)=0$ is needed for this equivalence). Consider now the
cone
\begin{equation}\label{eq:sigma}
\Sigma=\{w\in\R^n\st w^\T y\ge
0 \quad \forall y\in B(f)\}
\end{equation}
(that is, $\Sigma$ is the negative of the polar cone of $B(f)$).

\begin{remark}
The cone $\Sigma$ has a non-empty interior if and only if
$0\notin B(f)$.
\end{remark}
\begin{proof} Observe that $0\notin B(f)$ if and only if there exists an hyperplane weakly separating $0$ and $B(f)$, that is,
$0\notin B(f)$ if and only if there exists $w\in\R^n$ such that $w^\T y>0$ for all $y\in B(f)$, i.e., there exists $w$ is in the interior of $\Sigma$.
\end{proof}

A separation oracle for $\intr(\Sigma)$ is provided by \textsc{GreedyMin}, since for every
$y\in\R^V$ we have $y\in\intr(\Sigma)$ if and only if $\min_{x\in
  B(f)} y^\T x>0$.
Consequently, if the algorithm does not terminate in the required
running time bound, we can conclude that $f(S)\ge 0$ for all
$S\subseteq V$. We could use this algorithm in a binary search framework
to solve \eqref{prob:sfm}. When querying $\min_{S\subseteq V} f(S)\ge
-\mu$ for a $\mu>0$, we shift $f(S)$ by $f(S)+\mu$ for every
$S\subsetneq V$, $S\neq \emptyset$.

The main drawback of the binary search scheme is that it only provides the
optimum value, but does not give either an optimal set $S$, nor a dual
certificate as in Theorem~\ref{thm:edmonds}. Also, the
 binary search leads to an extra $\log \Fm{f}$ factor  in the running
 time.

In this section, we describe a variant of this algorithm, which
provides a primal optimal solution, and does not require binary
search. This will be achieved by dynamically shifting or ``sliding'' the function $f$
throughout the algorithm, as explained below. However, the algorithm
does not directly return a dual certificate of optimality. This can be
obtained using the pull-back technique introduced
in Section~\ref{sec:farkas}; see also the remark after
Theorem~\ref{thm:oracle-running time}.

We start by describing the sliding framework.
Besides the geometric rescaling algorithm described next, this
technique will also be useful for devising simple cutting plane
algorithms for SFM in Section~\ref{sec:ellipsoid}.

\paragraph{Sliding the function}
Throughout the algorithm, we  maintain a value $\mu\in\Z_+$, along
with a set $W$, such that $f(W)=-\mu$.
We initialize
$\mu=\max\{0,-f(V)\}$, and set $W=\emptyset$ or $W=V$ accordingly.
Hence $-\mu$ gives an upper bound on
$\min_{S\subseteq V} f(S)$.
The algorithm terminates once it
concludes that $f(W)=\min_{S\subseteq V} f(S)$ for the current $W$.
We define the function $f_\mu:2^V\to\Z$ as
\begin{equation}\label{eq:f-mu}
f_\mu(S)\eqdef\begin{cases}
0,\quad&\mbox{if }S=\emptyset \mbox{ or }S=V,\\
f(S)+\mu,\quad &\mbox{otherwise}.
\end{cases}
\end{equation}
This operation is known as the $\mu$-enlargement of the function $f$ (see
Fujishige \cite[Section 3.1(d)]{fujishige-book}). The operation has been
used in the context of submodular function minimization in~\cite{Fujishige2002}.

\begin{lemma}\label{lem:f-mu}
For a submodular function $f$ and a value $\mu\ge \max\{0,-f(V)\}$, the function
$f_\mu$ is submodular. If $0\in B(f_\mu)$, then $-\mu\le f(S)$ for
every $S\subseteq V$. Furthermore, $B(f_\mu)\subseteq B(f_{\mu'})$
whenever $\mu\le \mu'$.
\end{lemma}
\begin{proof}
The function $f'$ defined by $f'(S)=f(S)+\mu$ for all $S\subseteq V$ is clearly submodular. We obtain $f_\mu$
from $f'$ by decreasing the value of $f'(\emptyset)$ and $f'(V)$; note
that the bound on $\mu$ guarantees that these are both
nonnegative. Submodularity is maintained, since for any choice of $X$
and $Y$, the right-hand-side in
\eqref{eq:submod} decreases by at least as much as the left-hand-side when
replacing $f'$ by $f_\mu$. If $0\in B(f_\mu)$, then $0\le f_\mu(S)$
for any $S\subseteq V$. If $S\notin\{\emptyset,V\}$, then this gives
$f(S)\ge -\mu$; the choice of $\mu$ guarantees the same for
$S=\emptyset$ and $S=V$. For $\mu'\ge \mu$,  the containment $B(f_\mu)\subseteq
B(f_{\mu'})$ follows, since the constraints $x(S)\le f_{\mu'}(S)$ are
implied by the constraints $x(S)\le f_\mu(S)$.
\end{proof}

The following Lemma will be used to update the value of $\mu$.
\begin{lemma}\label{lem:minset}
Consider a value $\mu\ge \max\{0,-f(V)\}$, and
let $w:V\to \R$ be a cost function such that $\min\{w^\T x\colon
x\in B(f_\mu)\}>0$.
For $S=\textsc{MinSet}(f_\mu,w)$,
we have
$f(S)<-\mu$.
\end{lemma}
\begin{proof}
Let $v_1,\ldots,v_n$ be the ordering of $V$ returned by
\textsc{GreedyMin}$(f_\mu,w)$ such that  $w(v_1)\le w(v_2)\le\ldots\le w(v_n)$.
From \eqref{eq:vertex-cost} and from the fact that $f_\mu(V)=0$ (by construction) we see that the minimum value of
$w^\T x$ over $B(f_\mu)$ can be written as
\[
w^\T x=\sum_{i=1}^{n-1} (f(\{v_1,\ldots,v_i\})+\mu)(w(v_i)-w(v_{i+1})).
\]
Since $w^\T x>0$ and $w(v_i)-w(v_{i+1})\le 0$ for $i=1,\ldots,n-1$,  it follows that  $f(\{v_1,\ldots,v_i\})<-\mu$ for some value of $i$, implying the claim.
\end{proof}

\begin{lemma}\label{lem:vector-length}
Consider a value $\mu\ge \max\{0,-f(V)\}$ such that $\mu=-f(W)$ for
some $W\subseteq V$.
Then, $\Lb{f_\mu}\le 4\Lb{f}$.
\end{lemma}
\begin{proof}
For any permutation of the ground set, let $g$ and $g'$ be the
corresponding extreme bases in $B(f)$
and in $B(f_\mu)$, respectively.
These only differ in the first and last components: respectively by
$+\mu$, and by $-\mu-f(V)$. Hence, $\|g'\|_1\le \|g\|_1+2\mu+|f(V)|$. Note that
$\mu\le \Lb{f}$; this is because $\mu=-f(W)$ for a certain set $W$,
and therefore any permutation that starts with the elements of $W$
will give an extreme base of $1$-norm at least $|f(W)|$. Similarly, $|f(V)|\le \Lb{f}$.
 The claim
follows.
\end{proof}

\subsection{The sliding von Neumann algorithm}

\renewcommand{\algorithmicrequire}{\textbf{Input:}}
\renewcommand{\algorithmicensure}{\textbf{Output:}}

\begin{figure*}[htb!]
\begin{center}
\begin{minipage}{0.85\textwidth}
\begin{algorithm}[H]
\raggedright
  \begin{algorithmic}[1]
    \Require{A submodular function $f:2^V\to\Z$, a value
      $\mu\ge\max\{0,-f(V)\}$, a set $W\subseteq V$ with $f(W)=-\mu$, a
      matrix $Q\in \bb{S}^{n}_{++}$, and an $\varepsilon>0$.}
    \Ensure{$\,$
    \begin{itemize}
    \item A value $\mu'\ge \mu$ and a set $W'\subseteq V$ with
      $f(W')=-\mu'$,
    \item bases $g_1,\ldots,g_k\in B(f_{\mu'})$,
 $x\in \R^k$, $y\in \R^n$  such that
      $y=\sum_{i=1}^k{x_i g_i/\|g_i\|_Q}$, $\vec e^\T x=1$, $x\ge
      0$, and  $\|y\|_Q\le \varepsilon$.
      \end{itemize}
      }

    \State Set $\mu':=\mu$, $W':=W$.
    \State Pick $g_1$ as an arbitrary vertex of $B(f_\mu)$. Set $x_1:=1$, $y:={g_1}/{\|g_1\|_Q}$.
\State Let $k:=2$.
    \While{$\|y\|_Q>\varepsilon$}
  	\State Let $g_k\gets$ \Call{GreedyMin}{$f_{\mu'},Qy$}.
        \If{$y^\T Q g_k>0$} \Comment{sliding} \label{li:vN-sliding}
        \State \label{li:W} $W':=$\Call{MinSet}{$f_{\mu'},Qy$}; $\delta:=-f_{\mu'}(W')$;  $\mu':=-f(W')$;

        \State Set $v_1$ and $v_n$ to be the first and last elements
        of $V$ in increasing order by the weight vector $Qy$.
        \State \label{li:new g} $g_k(v_1):=g_k(v_1)+\delta$;
        $g_k(v_n):=g_k(v_n)-\delta$.
\EndIf
       \State {\bf end if}
      \State
      \[\lambda:=\frac{\scalar{y-\frac{g_k}{\|g_k\|_Q},y}_Q}{\left\|y-\frac{g_k}{\|g_k\|_Q}\right\|^2_Q}
      ;\] \label{l:lambda}
       \State $y:=(1-\lambda)y+\lambda g_k/\|g_k\|_Q$;
       \Comment{min $Q$-norm point on $[y,  g_k/\|g_k\|_Q]$}
       \State $x_k:=\lambda$;
\For{$i=1,\ldots,k-1$} $x_i:=(1-\lambda)x_i$\EndFor
\State $k:=k+1$
	\EndWhile
\Return{$\mu'$, $W'$, the vectors $g_1,\ldots,g_k$, $x$, and $y$.}
 \end{algorithmic}
\caption{The sliding von Neumann algorithm}\label{alg:Neumann}
\end{algorithm}
\end{minipage}
\end{center}
\end{figure*}

The Full Support Image Algorithm of \cite{part-1} uses
the von Neumann algorithm as the basic subroutine.
The von Neumann algorithm was described in \cite{Dantzig1991} to find a
feasible solution to the system $A^\T y>0$ for a matrix $A\in
\R^{n\times p}$. At every iteration, $y$ is maintained as a convex combination
of columns of $A$ normalized by their norm, that is, $y$ is maintained in the convex hull of $a_1/\|a_1\|,\ldots,a_p/\|a_p\|$.
Initially $y=a_i/\|a_i\|$ for some arbitrary $i\in[n]$; at any iteration, the algorithm terminates if $A^\T y>0$, otherwise
a column $a_k$ such that $a_k^\T y\leq 0$ is selected, and $y$ is updated to be the point of minimum norm in the line segment
$[a_k/\|a_k\|,y]$. The von Neumann algorithm can be seen as a variant of the conditional
gradient algorithm for the problem $\min\{\frac 12 \|y\|^2\st y\in\conv(\{a_1/\|a_1\|,\ldots,a_p/\|a_p\|\})\}$,
differing in the fact that von Neumann algorithm only needs to decide whether the minimum value of the norm is positive.

Our \emph{sliding von Neumann algorithm} (Algorithm~\ref{alg:Neumann}) is a modification of the standard von Neumann algorithm,
adapted to the context of submodular function minimization. The algorithm is applied to
the extreme bases of $B(f)$, in order to decide if there exists a point in the interior of the cone $\Sigma$ defined
in \eqref{eq:sigma}.
The main differences are the following.
\begin{itemize}
  \item The algorithm incorporates the adaptive shifting $f_\mu$ described previously. In particular,  when the current $y$ satisfies $g^\T y>0$ for all $g\in B(f_\mu)$,
  the algorithm does not stop, but it determines $S\subseteq V$ with $f_\mu(S)<0$ as in Lemma~\ref{lem:minset}, it updates $\mu:=-f(S)$, and resumes from the current point $y$.
  \item Rather than maintaining $y$ as a convex combination of the extreme bases of $B(f_\mu)$ normalized by their $2$-norms, the algorithm will use a more general norm, defined by a symmetric positive definite matrix $Q$ given as part of the input. This is because the algorithm will be used as a subroutine of Algorithm~\ref{alg:main-weakly}, where the norm will be periodically
   rescaled.
 \end{itemize}

\begin{lemma}\label{lemma:sliding VN}
Algorithm~\ref{alg:Neumann} terminates in $\lceil 1/\varepsilon^2\rceil$
iterations, returning a correct output.
\end{lemma}
\begin{proof}
{\em Correctness.} At every iteration, the algorithm calls \textsc{GreedyMin}$(f_{\mu'},Qy)$ to
obtain an extreme base $g_k\in B(f_\mu)$ minimizing $y^\T Qx$ over
$B(f_\mu)$. If $y^\T Q g_k\leq  0$, then we perform a von Neumann update, that is, we replace
$y$ with the minimum $Q$-norm point on the line segment $\left[y,\frac{g_k}{\|g_k\|_Q}\right]$
(which is given by the choice of $\lambda$ in line~\ref{l:lambda}).

Consider now the case $y^\T Q g_k>0$. In this case, at line~\ref{li:W} the current value of $\mu'$ is updated to a new value, say $\mu''$.
According to Lemma~\ref{lem:minset}, the set $W'$ determined at Line~\ref{li:W} satisfies $f_{\mu'}(W')<0$, hence
$\mu''=-f(W')=\mu'-f_{\mu'}(W')>\mu'$. In particular, observe that $\delta=\mu''-\mu'>0$. By definition, for all $S\subseteq V$,
\begin{equation}\label{eq:updated-fmu}
  f_{\mu''}(S)=\begin{cases} f_{\mu'}(S)+\delta & \mbox{if } S \neq \emptyset, V \\
                      0 & \mbox{if } S=\emptyset \mbox{ or } S= V \end{cases}
\end{equation}

The vector $g_k$ is updated at Line~\ref{li:new g}. Let us denote by $g'_k$ the updated vector.
It follows from \eqref{eq:vertex} and from \eqref{eq:updated-fmu} that $g'_k$ is the extreme base of $B(f_{\mu''})$ defined by the same ordering of the elements $v_1,\ldots,v_n$ of $V$ that defined $g_k$ as an extreme base of $B(f_{\mu'})$.

It follows that $g'_k$ is the solution output by \textsc{GreedyMin}$(f_{\mu''},Qy)$, hence it is an optimal solution for $\min\{y^\T Q z\st z\in B(f_{\mu''})\}$. Since, by the choice of $\mu''$, $f(\{v_1,\ldots,v_i\})\geq -\mu''$ for all $i\in\{1,\ldots,n\}$,  Lemma~\ref{lem:minset} implies that $y^\T Q g'_k\le 0$.

Since $g'_k$ is in $B(f_{\mu''})$ and  $y^\T Q g'_k\le 0$, we can perform a von Neumann update. Lemma~\ref{lem:f-mu} implies that $B(f_{\mu''})\supseteq B(f_{\mu'})$; hence all vectors $g_i$ computed thus far are still contained in $B(f_{\mu'})$. Since the algorithm terminates when $\|y\|_Q\leq \varepsilon$, it follows that the $t$ and the elements $g_1,\ldots,g_k$ returned
by the algorithm satisfy the requirements.

\smallskip

\noindent{\em Running time.} By the standard analysis of von Neumann's algorithm given by Dantzig~\cite{Dantzig1991}, $1/\|y\|_Q^2$ increases by at least $1$ at
every update, hence after $k$ iterations $\|y\|_Q\leq 1/\sqrt{k}$. We remark that, even though Dantzig's analysis applies to the case of the $2$-norm (i.e. $Q=I$), the case of a norm defined by a generic positive definite matrix $Q$ can be easily reduced to the $2$-norm case. Indeed, note that for any vector $z\in\R^n$, $\|z\|_Q=\|Q^{1/2}z\|_2$. In particular, since at every iteration we maintain $y=\sum_{i=1}^k{x_i g_i/\|g_i\|_Q}$, multiplying on both sides by $Q^{1/2}$ this is equivalent to $Q^{1/2}y=\sum_{i=1}^k{x_i (Q^{1/2}g_i)/\|Q^{1/2}g_i\|_2}$. Hence the steps of von Neumann algorithm for the norm defined by $Q$ are the same as the steps of the standard von Neumann algorithm applied to the vectors $(Q^{1/2} g)/\|Q^{1/2}g_i\|_2$, $g\in B(f)$, where $Q^{1/2}y$ is maintained as a convex combination of such vectors. Dantzig's analysis implies that $1/\|Q^{1/2}y\|^2_2=1/\|y\|_Q^2$ increases by at least $1$ at every update.
\end{proof}

\subsection{Geometric rescaling algorithm for SFM}\label{sec:rescaling}

In our geometric rescaling algorithm, Algorithm \textsc{Rescaling-SFM} shown in
Algorithm~\ref{alg:main-weakly}, we set  parameters
\[
\varepsilon\eqdef \frac{1}{20{n}},\qquad T\eqdef 5{n}\log (n \Lbt{f}).
\]

In the algorithm below we also maintain a set $\Gamma$ collecting all the elements of the base polytope generated during the Sliding von Neumann iterations; its main purpose is for clarity during the analysis, as we will need to refer to it.

\begin{figure*}[htb]
\begin{center}
\begin{minipage}{0.85\textwidth}
\begin{algorithm}[H]
\raggedright
  \begin{algorithmic}[1]
    \Require{A submodular function $f:2^V\to\Z$.}
    \Ensure{A set $W$ minimizing $f(W)$.}

   \State Set $Q:=I_{n}$, $R:=I_{n}$. Set $\Gamma:=\emptyset$.

   \State Set $\mu:=\max\{0,-f(V)\}$.
\If {$f(V)<0$} $W:=V$, \textbf{else} $W:=\emptyset$.  \EndIf
 \For{$i=1,\ldots,T$}
    \State Call \Call{Sliding von Neumann}{$f,\mu,W,Q,\varepsilon$} to obtain
    the new values

    of $\mu$ and $W$, and vectors $g_1,\ldots,g_k,x,y$.
    \State {\bf If} $y=0$, then {\bf stop}; {\bf return} $W$
    \State {\bf Else} {\bf rescale}
   \begin{equation}\label{eq:rescale}
R:=\frac{1}{(1+\varepsilon)^2}\left(R+\sum_{i=1}^k \frac{x_i}{\|g_i\|_Q^2}
  g_ig_i^\T \right);\quad  Q:=R^{-1}.
\end{equation}
\State $\Gamma:=\Gamma\cup\{g_1,\ldots,g_k\}$;
 \EndFor
\Return{$W$.}
 \end{algorithmic}
\caption{Rescaling-SFM}\label{alg:main-weakly}
\end{algorithm}
\end{minipage}
\end{center}
\end{figure*}

Algorithm \textsc{Rescaling-SFM} is the adaptation of the Full
Support Image Algorithm to our submodular
setting, using the sliding von Neumann algorithm. We need to modify
the algorithm and its analysis to reflect that the feasible region
keeps changing due to the updates to the value of $\mu$.

The value $\mu$ keeps increasing during the algorithm; it is updated
within the sliding von Neumann subroutine. We also maintain a set $W$ with $f(W)=-\mu$.
The algorithm stops after $T$ rescalings. At this point, we
conclude from a volumetric argument that the current $W$ is the
minimizer of $f$. We show the following running time bound.

\begin{theorem}\label{thm:main-weakly}
Algorithm \textsc{Rescaling-SFM} finds an optimal solution to \eqref{prob:sfm} in time
$O((n^4 \cdot \EO + n^5)\log (n\Lbt{f}))$.
\end{theorem}
Note that, the definition of $T$ requires knowing the value of
$\Lbt{f}$; we can replace it by the bound $\|\alpha\|_2$  as in
Section~\ref{sec:prelim}. As noted there, this changes the
overall running time bound only by a constant factor.

\begin{remark}\label{rmk:difference in R} The rescaling formula \eqref{eq:rescale} uses the
denominator $(1+\varepsilon)^2$, whereas in \cite{part-1} the denominator $1+\varepsilon$ is used instead.
This slightly different definition is needed in the proof of Lemma~\ref{lemma:pullback} in
Section~\ref{sec:farkas}. Nevertheless, the analysis in \cite{part-1}
goes through essentially unchanged by choosing, as we did here,  $\varepsilon=1/(20 n)$ instead of the choice $\varepsilon=1/(11 n)$ made in \cite{part-1}.
\end{remark}
In the analysis of Algorithm~\ref{alg:main-weakly}, we will refer to several lemmas from~\cite{part-1}. While the proofs are essentially given in~\cite{part-1}, they differ slightly from the current context due to the different choice of rescaling, mentioned in Remark~\ref{rmk:difference in R}, and the different notation. For completeness we will provide proofs in the appendix.

\subsection{Analysis}\label{sec:weakly-analysis}
Let us define the ellipsoid
\[
E(R)\eqdef \{x\in \R^n: x^\T R x\le 1\}.
\]
Further, let
\begin{equation}\label{eq:sigma-def}
\Sigma_\mu\eqdef \{w\in \R^{n}\colon w^\T x\ge 0 \quad \forall x\in
B(f_\mu)\},\qquad F_\mu\eqdef \Sigma_\mu\cap \B^{n},
\end{equation}
where $\B^n$ denotes the unit ball centered at the origin in $\R^n$.

$\Sigma_\mu$ is the set of normal vectors of hyperplanes that weakly separate 0 from
$B(f_\mu)$. A vector in the interior of $\Sigma_\mu$ gives a strong
separation, and verifies that $0\notin B(f_\mu)$. This in turn implies
that $f_\mu(S)<0$ for some set $S\subseteq V$, and thus the minimum
value of $f$ is strictly less than the current estimate $-\mu$.

 The main ideas of the analysis are showing that \emph{(a)} the
ellipsoid $E(R)$ contains the set $F_\mu$ at every iteration (Lemma
\ref{lem:F-in}), and that \emph{(b)} the volume of $E(R)$ keeps decreasing by a
constant factor at every rescaling (Lemma~\ref{lem:dual-vol-dec}). For
an integer valued $f$, one can lower bound the volume in terms of $n$
and $L_{f,2}$, assuming that $F_\mu$ has a nonempty interior. Hence,
at termination one can conclude that the interior of $F_\mu$ is empty,
which implies that $f_\mu\ge 0$, or equivalently, the minimum value of
the function is $-\mu$ for the current $\mu$.

The analysis below provides a slightly different argument than the
volume analysis, by bounding the $Q$-norm of the bases used during the
algorithm. This will be needed for the ``pull-back'' argument for
finding a dual certificate of optimality in Section~\ref{sec:farkas}.

Clearly, \textsc{GreedyMin}$(f_\mu,w)$ can be used as a
separation oracle for $\Sigma_\mu$. Further,
 if $\mu'\ge \mu$, then by Lemma~\ref{lem:f-mu} $B(f_{\mu'})\supseteq B(f_{\mu})$, which
 implies that   $\Sigma_{\mu'}\subseteq
\Sigma_{\mu}$ and $F_{\mu'}\subseteq F_\mu$ by definition.

As in \cite{part-1}, for a convex set $X\subset \R^{n}$ and a vector
$a\in \R^{n}$, we define the width
\[
\mathrm{width}_{X}(a)\eqdef\max\{a^\T z\colon z\in X\}.
\]
 Further, we define the condition number
\[
\omega_\mu \eqdef\min_{x\in B(f_\mu)\sm\{0\}}
\frac{\mathrm{width}_{F_\mu}(x)}{\|x\|_2}.
\]
A key estimate for the running time analysis is the following.
\begin{lemma}\label{lem:omega-bound}
Let $f$ be an integer valued submodular function, and let $\mu\in\Z$ satisfy
$\min_{S\subseteq V} f(S)< -\mu\le \min\{0,f(V)\}$.
Then
\[
\omega_\mu\ge \frac1{4 n\Lbt{f}}.
\]
\end{lemma}
\begin{proof}
Note that, by Claim~\ref{claim:complexity parameters} and Lemma~\ref{lem:vector-length},  $\Lbt{f_\mu}\le\Lb{f_\mu}\le 4\Lb{f}\le 4\sqrt{n}\Lbt{f}$, hence
$\|x\|_2\le 4\sqrt{n}\Lbt{f}$ for
every $x\in B(f_\mu)$. The claim follows by showing
\begin{equation}\label{eq:width-bound}
\mathrm{width}_{F_\mu}(x)\ge 1/\sqrt{n}.
\end{equation}
To prove this,  we note that the assumption of the
lemma implies $0\notin B(f_\mu)$. Let $z$ denote
the minimum norm point in $B(f_\mu)$, and let $\hat z=z/\|z\|_2$. Since $z$ is the minimum norm point in $B(f_\mu)$, we have $z^\T x\geq \|z\|^2_2$ for all $x\in B(f_\mu)$. Thus for all $x\in B(f_\mu)$
\[
\hat z^\T x\ge \|z\|_2,
\]
and, in particular, this implies $\hat z\in F_\mu$.
By Theorem~\ref{thm:fuji}, if $S$ is the minimizer of $f_\mu$, $f_\mu(S)\leq -1$, because $f$ is integer valued and
$\mu$ is a nonnegative integer such that $f(S)<-\mu$. It follows that
$1\le |f_\mu(S)|\le \sqrt{n}\|z\|_2$, thus $\hat z^\T x\ge
1/{\sqrt{n}}$. Since $\hat z\in F_{\mu}$, this provides the bound on
$\mbox{width}_{F_\mu}(x)$ for every $x\in B(f_\mu)$.
\end{proof}

\begin{lemma}\label{lem:R representation} At every stage of Algorithm~\textsc{Rescaling-SFM}, $\Gamma\subseteq B(f_\mu)$. Furthermore, after $t$ rescalings, the current matrix $R$ is of the form
\begin{equation}\label{eq:R-form}
R=\alpha I_n+\sum_{g\in\Gamma} \gamma_g\frac{g g^\T}{\|g\|_2},
\end{equation}
where $\alpha=\frac{1}{(1+\varepsilon)^{2t}}$ and $\gamma_g>0$ for all $g\in\Gamma$.
\end{lemma}

\begin{proof}
To see that $\Gamma\subseteq B(f_\mu)$ throughout the algorithm, it suffices to observe that when a new element $g$ is included in $\Gamma$, it is an extreme base for $B(f_\mu)$ for the current value of $\mu$, so in particular it is an element of $B(f_\mu)$, and that throughout the algorithm $B(f_\mu)$ can only become inclusionwise larger, as the value of $\mu$ never decreases. The second part of the statement follows immediately from the recursive construction of $R$ as defined in \eqref{eq:rescale}.
\end{proof}

Lemmas~\ref{lem:F-in}, \ref{lem:lower-width}, and \ref{lem:a-k-Q-bound} below follow essentially from \cite{part-1}, thus we will often refer to \cite{part-1} in the proofs. To enable the reader to check the arguments in  \cite{part-1}, we first explain how the notation differs between the two papers. Specifically, in

\begin{lemma}[{\cite[Lemma 10]{part-1}}]\label{lem:F-in}
Throughout Algorithm~\textsc{Rescaling-SFM}, $F_\mu\subseteq E(R)$ holds.
\end{lemma}
\begin{proof}
At initialization we have $F_\mu\subseteq \B^n=E(R)$, since initially $R=I_n$. During the algorithm the value of $\mu$ can increase (in the sliding von Neumann algorithm) or the matrix $R$ can be updated (rescaling). In the former case, as noted above, $F_\mu$
becomes smaller (with respect to set inclusion) as $\mu$ increases in value,  hence the property is maintained in these steps.

The proof that $F_\mu\subseteq E(R)$ whenever $R$ is updated is essentially identical to the proof of Lemma 10 in \cite{part-1}. For simplicity we provide a proof in Lemma~\ref{lemma:E(R) contained at rescaling} in Appendix~\ref{appendix}.
\end{proof}

\begin{lemma}~\label{lem:lower-width}
Throughout Algorithm~\textsc{Rescaling-SFM},
$\|x\|_Q\ge \omega_\mu\|x\|_2$ must hold for every $x\in B(f)\sm \{0\}$.
\end{lemma}
\begin{proof} Since $F_\mu\subseteq E(R)$ by Lemma~\ref{lem:F-in}, for every
$x\in B(f)\sm\{0\}$ we have $\mathrm{width}_{F_\mu}(x)\leq \mathrm{width}_{E(R)}(x)$. Furthermore
$\mathrm{width}_{E(R)}(x)=\|x\|_Q$ (this is an easy fact, see \cite[Lemma 2.15]{part-1}). The statement
now follows from the definition of $\omega_\mu$.
\end{proof}

\begin{lemma}[{\cite[Lemma 11]{part-1}}]\label{lem:dual-vol-dec}
The determinant of $R$ increases at least by a factor $16/9$ at
every rescaling.
\end{lemma}
\begin{proof}
Proof in Appendix~\ref{appendix}.
\end{proof}

\begin{lemma}[{\cite[Lemma 4.11]{part-1}}]\label{lem:a-k-Q-bound}
At any stage of Algorithm~\textsc{Rescaling-SFM}, there exists a point $\bar g\in \Gamma$ such that
\[\|\bar g\|_Q{\sqrt{\det(R)^{1/{n}}-1}}\le {\|\bar g\|_2}.\]
\end{lemma}
\begin{proof}
Proof in Appendix~\ref{appendix}.
\end{proof}

\begin{proof}[Proof of Theorem~\ref{thm:main-weakly}]
One possibility is that the algorithm terminates at step 6 because $y=0$. In this case, observe that $y=0$ is a conic combination of elements $g_1,\ldots,g_k\in B(f_\mu)$, so it is also a convex combination of those elements, thus showing that $0\in B(f_\mu)$. It follows that $\min_S f_\mu(S)=0$, hence $\min_S f(S)=-\mu=f(W)$, showing that  $W$ is a minimizer.

Otherwise, the algorithm performs $T=5n \log (n\Lbt{f})$ rescalings.
Lemma~\ref{lem:dual-vol-dec} shows that after $T$ rescalings,
$\det(R)\ge (16/9)^T$. It follows from Lemma~\ref{lem:a-k-Q-bound} that,
after $T$ rescalings, there exists a point $g_h\in B(f_\mu)$ with
$\|g_h\|_Q<\|g_k\|_2/(4n\Lbt{f})$. By Lemma~\ref{lem:lower-width} we
conclude that $\omega_\mu\leq \|g_k\|_Q/\|g_k\|_2< 1/(4n\Lbt{f})$. Noting that $\mu$ is maintained
to be an integer throughout the execution of the algorithm,
Lemma~\ref{lem:omega-bound} implies that $\min_{S\subseteq V} f(S)=-\mu$. Since the algorithm
maintains a set $W$ with $f(W)=-\mu$, we can conclude that $W$ is a minimizer for \eqref{prob:sfm}.
This shows that the algorithm terminates with a correct solution.

The algorithm calls the sliding von Neumann subroutine $T=O(n\log (n\Lbt{f}))$ times; by Lemma~\ref{lemma:sliding VN} each call
takes at most $\lceil 1/\varepsilon^2\rceil=O(n^2)$ iterations.
At the $k$th iteration of von Neumann, it takes time
$O(n\cdot\EO+n\log n)$ to run \textsc{GreedyMin} and time $O(k)$ to
update the coefficients $x_1,\ldots,x_k$. These give a bound of
$O(n^3\cdot\EO+n^4)$ for each sliding von Neumann subroutine.

Further, every rescaling has to compute $O(n^2)$ outer
products $g_ig_i^\T$, add their weighted sum to $R$, and compute
$Q=R^{-1}$. The computation is dominated by computing the outer
products, which take altogether $O(n^4)$ time. Hence the iterations
between two subsequent rescalings take time $O(n^3\cdot \EO+n^4)$, yielding
the claimed complexity bound.
\end{proof}


\subsection{Finding approximate solutions via the geometric rescaling algorithm}\label{sec:farkas}

The main purpose of this section is to show how \textsc{Rescaling-SFM} can be used to derive approximate solutions to \eqref{prob:sfm} in time that depends logarithmically on the desired precision. This will also lead to a strongly polynomial-time based on \textsc{Rescaling-SFM}, as explained in Section~\ref{sec:strongly}. Furthermore, as a bi-product, we will  also be able to generate dual certificates of optimality for \eqref{prob:sfm} from the output of \textsc{Rescaling-SFM}.

\begin{theorem}\label{thm:oracle-running time}
Setting $T=O( n\log( n\delta^{-1}))$ in
Algorithm~\ref{alg:main-weakly}, from its output one can compute  a
set $W\subseteq V$ and a point $y\in B(f)$, expressed as a convex
combination of $O( n)$ extreme bases of $B(f)$,
such that $f(W)\le y^-(V)+\delta\Lb{f}$. In particular, $f(W)\leq \min_{S\subseteq V} f(S) +\delta\Lb{f}$. The overall running time is $O(( n^4\cdot\EO+ n^5)\log( n\delta^{-1}))$.
\end{theorem}

To prove the above theorem we will use a \emph{``pull-back''} technique. Recall that in \textsc{Rescaling-SFM}, we keep modifying the matrix $Q$ defining the scalar product. Lemmas~\ref{lem:dual-vol-dec} and \ref{lem:a-k-Q-bound} guarantee that after $t$ rescalings, we can identify a vector $g\in B(f_\mu)$ that has a small $Q$-norm for the current $Q$, and the bound decreases geometrically with $t$. Our key technical claim, Lemma~\ref{lemma:pullback}, shows a constructive way to identify a vector $v\in B(f_\mu)$ with $\|v\|_2\le \|g\|_Q$. Provided a vector $v$ with small 2-norm (and thus small 1-norm), we can easily satisfy the requirements of \textsc{Approx-SFM}, using the following lemma.

%

\begin{lemma}\label{lem:slide-back}
Let $\mu\geq\max\{0,-f(V)\}$ and $W\subseteq V$ such that $f(W)=-\mu$. Let $\Gamma\subseteq B(f_\mu)$ finite and $\{\mu_g\}_{g\in\Gamma}\subseteq  [0,\mu]$ be such that $g$ is an extreme base of  $B({f_{\mu_g}})$ for all $g\in\Gamma$. Given $v=\sum_{g\in\Gamma} \lambda_g g$ where $\lambda_g\geq 0$, $g\in\Gamma$, and $\sum_{g\in\Gamma} \lambda_g=1$, in time $O( n |\Gamma|)$ we can compute $y\in B(f)$, given as a convex combination of $|\Gamma|$ extreme bases of $B(f)$, such that $$f(W)\leq y^-(V)+\frac{\|v\|_1}{2}.$$
\end{lemma}
\begin{proof}
For all  $g\in\Gamma$, let $\tilde g$ be the extreme base of $B(f)$ defined by the same ordering which defined the extreme base $g$ of $B({f_{\mu_g}})$.
Define $y:=\sum_{g\in\Gamma} \lambda_g \tilde g$.

Observe that, given $g\in\Gamma$, if $v_1,\ldots,v_ n$ is the ordering defining $g$, then $\tilde g(v_1)=g(v_1)-\mu_g$, $\tilde g(v_j)=g(v_j)$ for $j=2,\ldots, n-1$, and $\tilde g(v_n)= g(v_ n)+\mu_g+f(V)$. Thus, computing $\tilde g$ for all $g\in\Gamma$ requires time $O(|\Gamma|)$ and computing $y$ requires time $O( n |\Gamma|)$. Furthermore, we have that $\|\tilde g-g\|_1=2\mu_g+f(V)\le 2\mu+f(V)$.
This implies that
$$\|y\|_1=\|\sum_{g\in\Gamma} \lambda_g \tilde g\|_1\le \|\sum_{g\in\Gamma} \lambda_g g\|_1+\sum_{g\in\Gamma} \lambda_g \|\tilde g-g\|_1\le \|v\|_1+2\mu+f(V)=\|v\|_1-2f(W)+f(V).$$
Since $\|y\|_1=f(V)-2y^-(V)$, the above implies
$$f(W)\leq \frac{\|v\|_1+f(V)-\|y\|_1}{2}=y^-(V)+\frac{\|v\|_1}{2}.$$
\end{proof}

Our next lemma enables pulling back a  vector with small $Q$-norm to a vector with no larger 2-norm. This is done gradually, by pulling back at each rescaling of \textsc{Rescaling-SFM}. The lemma is not specific to the context of submodular function minimization. In our application, the columns of the matrix $A$ will be the bases used in the current iteration of the sliding von Neumann algorithm.
We also note that this technique is applicable to the general Full Support Image Algorithm in \cite{part-1}, enabling to find approximate solutions as well as dual certificates of infeasibility.

\begin{lemma}
\label{lemma:pullback} Let $A\in\R^{ n\times p}$, $R\in\bb{S}_{++}^n$,  $Q=R^{-1}$, and $\varepsilon>0$. Let $x\in\R^p_+$ such that $y:=\sum_{i=1}^p x_i \frac{a_i}{\|a_i\|_Q}$ satisfies $\|y\|_Q\leq \varepsilon$. Define
\begin{equation}\label{eq:rescale-pullback}
R'\eqdef \frac{1}{(1+\varepsilon)^2}\left(R+\sum_{i=1}^p \frac{x_i}{\|a_i\|_Q^2} a_i a_i^\T\right),
\end{equation}
and $Q'\eqdef (R')^{-1}$. For every $v\in \R^{n}$, there exists  $\nu\in\R^p_+$ such that $\|v+A\nu\|_Q\leq \|v\|_{Q'}$.
Moreover, such a vector $\nu$ can be computed in time $O(n^2 p)$.
\end{lemma}
\begin{proof}
For any given $v\in \R^{n}$, we define  $u\eqdef\displaystyle\frac{1}{(1+\varepsilon)^2}RQ'v$ and let
\begin{equation}\label{eq:lambda multi-rank}\beta\eqdef\max_{i\in [p]} \frac{\scalar{a_i,u}_Q}{\|a_i\|_Q},\quad
\nu_i\eqdef \frac{x_i}{\|a_i\|_Q} \left(\beta-\frac{\scalar{a_i,u}_Q}{\|a_i\|_Q}\right)\quad \mbox{for } i\in [p].\end{equation}
We will show that the statement is satisfied by the choice of $\nu\in\R^p_+$ defined above. These values can be clearly computed in $O(n^2 p)$ time.

First, we observe that, by substituting the definitions of $R'$ and $u$, we obtain
$$v=R'QRQ'v=u+\sum_{i=1}^p x_i \frac{\scalar{a_i,u}_Q}{\|a_i\|_Q} \frac{a_i}{\|a_i\|_Q},$$
which, from the definition of $\nu$ and $\beta$, implies that
\begin{equation}\label{eq:v+A mu} v+A\nu=u+\beta y.\end{equation}

Next, notice that
\begin{eqnarray*}
\|v\|_{Q'}&=&\sqrt{ v^\T Q' R' Q' v}=\frac{1}{(1+\varepsilon)}\left((v^\T Q' )R (Q'v)+v^\T Q' \left(\sum_{i=1}^p \frac{x_i}{\|a_i\|^2_Q} a_i a_i^\T\right) Q' v\right)^{\frac 12}\\
&\geq &\frac{1}{(1+\varepsilon)}\left((v^\T Q' R)Q (RQ'v)\right)^{\frac 12} =
(1+\varepsilon)\|u\|_Q.
\end{eqnarray*}

From the above and observing that $|\beta|\leq \|u\|_Q$, from the definition of $\beta$, we have
$$\|v+A\nu\|_Q\leq \|u\|_Q+|\beta| \|y\|_Q\leq (1+\varepsilon)\|u\|_Q\leq  \|v\|_{Q'},$$
where the first inequality follows from \eqref{eq:v+A mu} and the triangle inequality.
\end{proof}

\begin{remark}
  \label{rmk:caratheodory} If a vector $y$ is expressed
  as a convex combination of $\kappa$ extreme bases, then in time $O(n^2 \kappa)$ we can reduce the number of extreme bases in the convex combination to $O(n)$ by Carath\'eodory's theorem.
  \end{remark}

We are ready to prove Theorem~\ref{thm:oracle-running time}, showing how \textsc{Approx-SFM} can be implemented using \textsc{Rescaling-SFM}.

\begin{proof}[Proof of Theorem~\ref{thm:oracle-running time}.]
For the second part of the statement, note that if we are given $y\in B(f)$ and $W\subset V$ such that $f(W)\le y^-(V)+\delta L_f$, it follows from Theorem~\ref{thm:edmonds} that  $f(W)\leq \min_{S\subseteq V} f(S) +\delta$.

We now describe how to find such a vector $y$ and set $W$. Run algorithm \textsc{Rescaling-SFM}($f$), setting the limit on the number of rescalings to a number $T=c{n}\log(n\delta^{-1})$ for some constant $c$ to be specified later. At the end of the execution, we identified a value $\mu>0$ and a set $W\subseteq V$ such that $f(W)=-\mu$. Let $\Gamma$ be the set of points in $B(f_\mu)$ maintained throughout the algorithm.

By Lemma~\ref{lem:a-k-Q-bound}, for an appropriate choice of $c$, after $T$ rescalings there exists $\bar g\in\Gamma$ such that
$$\|\bar g\|_Q\le\frac{2\delta}{4\sqrt{n}}\|\bar g\|_2.$$  Let $\hat g=\bar g/\|\bar g\|_2$. The running time of \textsc{Rescaling-SFM}($f$) with the above choice of $T$ is  $O(({n}^4\EO+{n}^5)\log(n\delta^{-1}))$. Note also that $|\Gamma|\in O({n}^3 \log(n\delta^{-1}))$, thus finding $\bar g$ requires time $O({n}^5 \log(n\delta^{-1}))$ to compute the $Q$-norms of all elements of $\Gamma$.

By  applying Lemma~\ref{lemma:pullback} for $T$ times (considering the rescaling matrices used in the algorithm in reverse order), we can find a vector $\nu\in\R^\Gamma_+$ such that $\|\hat g+\sum_{g\in\Gamma}\nu_g g \|_2\leq \|\hat g\|_Q$.  Recall that each rescaling matrix is defined by at most ${n}^2$ vectors in $\Gamma$, therefore each application of Lemma~\ref{lemma:pullback} requires time $O(n^4)$ (assuming that the matrices $Q$ and $R$ used at every rescaling are saved in memory so we do not need to recompute them). Thus, overall,  the time required to compute $\nu$ is  $O({n}^5\log({n} \delta^{-1}))$.

Define $\alpha=1+\|\bar g\|_2\cdot \sum_{i=1}^{h}\nu_g$, and $\lambda\in\R_+^\Gamma$ by
$$
\lambda_g=\left\{\begin{array}{ll}
 \displaystyle        \frac{\|\bar g\|_2\nu_g}{\alpha},& g\in\Gamma\sm\{\bar g\} \\
 \displaystyle  \frac{1+\|\bar g\|_2\nu_{\bar g}}{\alpha}, & g=\bar g
       \end{array}\right.
$$
Define $v:=\sum_{g\in\Gamma}\lambda_g g$. Observe that $\sum_{g\in\Gamma}\lambda_g=1$, thus $v\in B(f_\mu)$. Computing $v$ requires time $O({n}^4 \log(n\delta^{-1}))$, since we need to sum $|\Gamma|$ ${n}$-dimensional vectors.

Furthermore,
$$\|v\|_1\leq \sqrt{n}\|v\|_2=\sqrt{n}\frac{\|\bar g\|_2}{\alpha}\left\|\hat g+\sum_{g\in\Gamma}\nu_g g \right\|_2\leq \sqrt{n}L_{f_\mu,2}\|\hat g\|_Q\leq 2\delta L_{f},$$
where the last inequality follows from the fact that $L_{f_\mu,2}\leq L_{f_\mu}\leq 4L_{f}$ by Lemma~\ref{lem:vector-length}. By Lemma~\ref{lem:slide-back}, in time $O({n}^4\log({n}\delta^{-1}))$ we can compute $y\in B(f)$ satisfying $f(W)\leq y^-(V)+\frac{\|v\|_1}{2}\leq y^-(V)+\delta L_f$.

Finally, since $y$ is expressed as a convex combination of $|\Gamma|\in O(n^3\log({n}\delta^{-1}))$ extreme bases, by Remark~\ref{rmk:caratheodory} we can express $y$ as a convex combination of $O({n})$ extreme bases  in time $O(n^5\log({n}\delta^{-1}))$.
\end{proof}

\begin{remark}\em The bound $O({n}^5 \log(n\delta^{-1}))$ for computing $\nu$ in the above proof was assuming $O({n}^2)$ time for computing $Q$-scalar products $\scalar{g,u}_Q$. We note that this can be easily improved by a factor $n$: we can assume that $Qg$ was precomputed and stored during \textsc{Rescaling-SFM} for all bases $g$ used during the sequence of rescalings. Indeed, it was necessary to compute the norms $\|g\|_Q$ in the sliding von Neumann algorithm. Thus, the bound improves to $O({n}^4 \log(n\delta^{-1}))$; however, this does not change the overall running time estimate.
\end{remark}

\paragraph{Finding a dual certificate in Rescaling-SFM}
For an integer valued $f$, a pair $W$ and $y$ satisfying the requirements of \textsc{Approx-SFM}$(f,1/\Lb{f})$ are an optimal pair of primal and dual solutions as in Theorem~\ref{thm:edmonds}. Hence the algorithm of Theorem~\ref{thm:oracle-running time} for $\delta=1/\Lb{f}$ provides a dual certificate of optimality in time $O(({n}^4\cdot \EO + {n}^5)\log ({n}\Lb{f}))$, the same as the complexity bound as in Theorem~\ref{thm:main-weakly} (using that $\Lb{f}\le \sqrt{n}\Lbt{f}$).

\paragraph{Approximate infeasibility certificates for conic linear programs}

Consider the feasibility problem for conic linear programs of the form
\begin{equation}\label{eq:conic}
A^\T z>0,
\end{equation}
where $A$ is a $p\times n$ matrix of rank $n$, whose columns $a_1,\ldots,a_p$ are assumed to have unit $\ell_2$-norm. It is well known that \eqref{eq:conic} is feasible if and only if $0$ is not contained in the convex hull of $a_1,\ldots,a_p$. Alternatively, one can say that \eqref{eq:conic} is infeasible if and only $\min\{\|v\|_2\st v\in\conv(A)\}=0$. For any $\delta>0$, Lemma~\ref{lemma:pullback} allows us to compute in time $O(n^2 p^2\log(\delta^{-1}))$, via the Image Algorithm of \cite{part-1},
either a solution to \eqref{eq:conic}, or a {\em $\delta$-approximate certificate of infeasibility}, that is, a point $v\in\conv(A)$ such that $\|v\|_2\leq \delta$.

Just as Algorithm~\ref{alg:main-weakly} in this paper, the Image Algorithm maintains matrices $R, Q\in\bb{S}_{++}^n$, $Q=R^{-1}$, which are updated at every rescaling. Initially $R=Q=I_n$. Between each rescaling, the algorithm uses von Neumann's method to compute, in at most $\ceil{\varepsilon^{-2}}$ iterations, a vector $x\in\R^p_+$, $\vec{e}^\T x=1$, such that the point $y:=\sum_{i=1}^p x_i \frac{a_i}{\|a_i\|_Q}$ satisfies either $A^\T Qy>0$ or $\|y\|_Q\leq \varepsilon$, for $\varepsilon:=1/(20n)$. In the former case, the algorithm stops since $z:=Qy$ is a feasible solution to \eqref{eq:conic}. In the latter case, $R$ is replaced with the matrix $R'$ given in \eqref{eq:rescale-pullback}, and $Q=R^{-1}$ is recomputed. After $T=O(n\log(\delta^{-1}))$ rescalings, by Lemmas~\ref{lem:dual-vol-dec} and~\ref{lem:a-k-Q-bound} there exists a column $a_k$, $k\in [p]$, such that $\|a_k\|_Q\leq \delta$. By  applying Lemma~\ref{lemma:pullback} for $T$ times (considering the rescaling matrices used in the algorithm in reverse order), we can find a vector $\nu\in\R^p_+$ such that $\|a_k+A\nu \|_2\leq \|a_k\|_Q\leq \delta$. If we define $$v=\frac{a_k+A\nu}{1+\|\nu\|_1},$$ then $v\in\conv(A)$ and $\|v\|_2\leq \delta/(1+\|\nu\|_1)\leq \delta$. Each call of von Neumann algorithm requires $O(np^2)$ arithmetic operations
(see proof of Theorem~3.2 in \cite{part-1}), whereas each application of Lemma~\ref{lemma:pullback} requires $O(n^2p)$ operations. It follows that the overall running time of the Image algorithm is $O(n^2p^2\log(\delta^{-1}))$.


\section{Strongly polynomial algorithms}\label{sec:strongly}
In this section, we provide a general scheme to convert an approximate
SFM algorithm to a strongly polynomial one. We assume that the SFM
algorithm is provided via the following oracle.

\begin{center} \fbox{\begin{minipage}{0.7\textwidth} \noindent
{\sf Oracle } {\sc Approx-SFM} \\
{\bf Input:}  \ A submodular function $f:2^V\to\mathbb{R}$ and $\delta>0$.\\
{\bf Output:}  A set $W\subseteq V$, and a vector $y\in B(f)$ such that
\[
f(W)\le y^-(V) + \delta \Lb{f}.
\]
Further, $y$ is given as a convex combination of $O(n)$ bases of
$B(f)$.
\end{minipage}}\end{center}
\medskip

Recall from
Theorem~\ref{thm:edmonds} that $f(S)\ge y^-(V)$ for any $S\subseteq
V$. Consequently,  the set $W$ returned by the oracle is within $\delta
L_f$ from the optimal solution to \eqref{prob:sfm}.

The output of the oracle will be used via
the following lemma.
\begin{lemma}\label{lem:contract-element}
Let $y\in B(f)$ and $W\subseteq V$ such that $f(W)\le y^-(V)+\delta L_f$. If
$y(v)<-\delta\Lb{f}$, then $v$ must be contained in every
minimizer of $f$.
\end{lemma}
\begin{proof}
Let $S\subseteq V\setminus \{v\}$. Then
$f(S)\ge y(S)\ge y^-(V\setminus\{v\})
\ge f(W)-y(v) -\delta \Lb{f}
> f(W).
$
This shows that $S$ cannot be an optimal solution to
\eqref{prob:sfm}.
\end{proof}
For a $y\in B(f)$ returned by the approximate oracle and an element $v$
as in the lemma, minimizing $f$ can be reduced to minimizing the contraction $f':2^{V\setminus\{v\}}\to
\mathbb{Z}$, defined as $f'(S)\eqdef f(S\cup\{v\})-f(\{v\})$.

The existence of such a $v$ will not always be
guaranteed. If we cannot immediately reduce the size of the ground set
by contraction, we will still be able to identify
structural properties of optimal solutions using the next simple lemma:
\begin{lemma}\label{lem:extend-U}
Let $y\in B(f)$, $U\subseteq V$, and $v\in V\setminus U$.
Assume that $y(v)>-y^-(V\setminus U)$.
Then any minimizer to \eqref{prob:sfm} that contains  $v$ must contain some element of $U$.
\end{lemma}
\begin{proof}
Let $S\subseteq V\setminus U$, $v\in S$. Then
$f(S)\ge y(v)+y^-(V\setminus U)> 0=f(\emptyset)$, hence $S$ cannot be a minimizer.
\end{proof}

During the algorithm we will maintain a partial order $\preceq$, where
$u\preceq v$ means that every minimizer to \eqref{prob:sfm} containing $v$ must
also contain $u$. New relations will be derived using the lemma above.
To guarantee either a contraction or a new relation in every
iteration, we need to call the oracle {\sc Approx-SFM} on a modified
version $\lf$ instead of the original $f$, introduced in the next subsection.

Section~\ref{sec:ring}
introduces the  machinery and notation. The strongly polynomial
scheme, Algorithm~\ref{alg:strongly-1}, is given in
Section~\ref{sec:strong-and-simple}.
It repeatedly calls {\sc Approx-SFM} with $\delta=1/(3n^3)$ and
contracts elements of the ground set, or learns new relations
concerning the optimal solutions to \eqref{prob:sfm}.
 Theorem~\ref{thm:strongly-1} gives the running time bound that is
 polynomial in $n$ and the running time of the approximate oracle
 calls. Section~\ref{sec:strong-fast} presents an enhanced version of
 the algorithm with improved running times.

\paragraph{Approximate oracle implementations}
The oracle  {\sc Approx-SFM} has been used implicitly or explicitly in
several papers in the literature:
\begin{itemize}
\item   Our geometric rescaling algorithm
  (Algorithm~\ref{alg:main-weakly}), in time
  $O((n^4\cdot\EO+n^5)\log(n\delta^{-1}))$, as explained in
  Theorem~\ref{thm:oracle-running time} in Section~\ref{sec:farkas}.
\item The conditional gradient method. Theorem~\ref{thm:cond-grad}
  asserts that
 in time
    $O((n^2\cdot\EO+n^2\log n)\delta^{-2})$, we can obtain a suitable
    $y$ as a
    convex combination of $O(n/\delta^2)$ extreme bases. According to Remark~\ref{rmk:caratheodory}, we need $O(n^3/\delta^2)$ arithmetic operations to reduce the support to $O(n)$ extreme bases. This gives an overall running time of
      $O((n^2\cdot\EO+n^3)\delta^{-2})$.
  \item The Fujishige-Wolfe algorithm in time
    $O((n^2\cdot\EO+n^4)\delta^{-2})$ (Theorem~\ref{thm:wolfe}).
  \item The Iwata-Fleischer-Fujishige weakly polynomial scaling algorithm~\cite{IFF}, in time $O(n^5\cdot EO\log(n\delta^{-1}))$.\footnote{The authors give both a weakly and a strongly polynomial algorithm; here we are referring to the weakly polynomial algorithm of Section 3 in~\cite{IFF}. The observation we reported here is explicitly stated by the authors in the last paragraph of Section 3 in~\cite{IFF}.}
  \item The Iwata-Orlin weakly polynomial
    algorithm~\cite{Iwata-Orlin}, in time
    $O((n^4\cdot\EO+n^5)\log(n\delta^{-1}))$.\footnote{This is not
      explicitly stated in~\cite{Iwata-Orlin}. However, the analysis in that paper
      shows that, in time $O((n^4\cdot\EO+n^5)\log(n\delta^{-1}))$,
      one obtains a set $W\subseteq V$ and a point $x\in B(f)$ such
      that $x(W)=f(W)$, $x(v)\geq 0$ for all $v\in V\sm W$, and
      $\Phi(x):=\sum_{v\in W}(x^+(v))^2\leq \delta^2\Lb{f}^2/n$. This
      implies that $f(W)=x^-(W)+x^+(W)\leq x^-(V)+\sqrt{n\Phi(x)}\leq
      x^-(V) + \delta \Lb{f}$.}
\item the Sidford-Lee-Wong cutting plane method in $O(n^2 \cdot \EO \log(n\delta^{-1})+n^3 \log^{O(1)}
(n\delta^{-1}))$; this is shown in Section~\ref{sec:ellipsoid}.
\end{itemize}

\subsection{Ring families and the structure of optimal solutions}\label{sec:ring}

From every call to {\sc Approx-SFM}, our algorithm will either
find an element $v\in V$ as in Lemma~\ref{lem:contract-element}, or at least one
pair $(v,u)$ such that every
optimal solution to \eqref{prob:sfm} containing $v$ must also contain
$u$; these pairs will be derived using Lemma~\ref{lem:extend-U}.
In the former case, we can simply reduce the size of the ground set by
contracting such an element. To make use of the pairs $(v,u)$ obtained
in the latter case, we use the following framework, first introduced by Iwata,
Fleischer, and Fujishige \cite{IFF}, and subsequently used in
several other algorithms, e.g. in \cite{Iwata-Orlin,LSW}.

\paragraph{Directed graph and partial order}
We maintain a directed graph $D=(V,F)$, with the property that if $(u,v)\in F$,
and $S\subseteq V$ is an optimal solution to \eqref{prob:sfm} with $u\in
S$, then $v\in S$. In other words, for every minimizer $S\subseteq V$ to \eqref{prob:sfm}, $\delta^+_F(S)=\emptyset$.

We can assume that $D$ is a acyclic. Indeed, given the set
of vertices $C$ of any strongly connected component of $D$, every
minimizer of $f$ must either contain $C$ or be disjoint from $C$,
hence we can contract all vertices of $C$ into a single new node
$v_C$. That is, we replace $V$ with $V':=V\cup\{v_C\}\sm C$, and  $f$ with the function $f'\st V'\to \R$ defined by
\begin{equation}\label{eq:contract}
f'(S)\eqdef\begin{cases}
f(S),\quad&\mbox{if }S\subseteq V'\sm\{v_C\},\\
f(S\cup C\sm \{v_C\}),\quad &\mbox{if } v_C\in S\subseteq V'.
\end{cases}
\end{equation}

The acyclic graph $D=(V,F)$ defines a partial order
$\preceq_F$. We have $u\preceq_F v$ if there exists a directed path in
$F$ from $v$ to $u$ (so, in particular, $v\preceq_F v$ for all $v\in V$).
We say that an ordering of the vertices is {\em consistent} with
$\preceq_F$, if $u$ is ordered before $v$ whenever $u\preceq_F v$. We
use the simpler notation  $u\preceq v$ if $F$ is clear from the context.

The {\em transitive closure} of the acyclic graph $D=(V,F)$ is the
graph $\bar D=(V,\bar F)$, where $(u,v)\in \bar F$ if
and only if there is a directed path from $u$ to $v$ in $F$. $D$ is said to be {\em transitive} if it coincides with its transitive closure. Clearly $D$ and its transitive
closure define the same partial order $\preceq_F$. We maintain the digraph $D=(V,F)$ transitive throughout the algorithm.

\paragraph{Ring families}
In terms of the partial order $\preceq$, every optimal solution to
\eqref{prob:sfm} is {\em a lower ideal} in the poset
$(V,\preceq)$. Let ${\cal F}$ denote the set of all lower ideals,
namely, $S\in {\cal F}$ if $v\in S$ and $u\preceq v$ implies $u\in S$.
Equivalently, in  terms of the digraph $D=(V,F)$, we let
\begin{equation}\label{eq:family F}
{\cal F}:=\{X\subseteq V\st \delta^+_F(X)=\emptyset\}.
\end{equation}
Thus, the family $\cal F$ contains all minimizers.
The family ${\cal F}\subseteq 2^V$ is  a {\em ring family}, that is, for every $X,Y\subseteq V$,
if $X,Y\in {\cal F}$ then $X\cap Y, X\cup Y\in {\cal F}$.

We note that one can define submodular functions over ring families
instead of the entire $2^V$. Submodular function minimization over
ring families has been well-studied and can be reduced to
submodular function minimization over $2^V$, see \cite[Section
3]{fujishige-book} and \cite[Chapter 49]{Schrijver03}.

The following definitions and results are similar to those in
\cite[Section 49.3]{Schrijver03}. For any $v\in V$, let
\[
\Down{v}\eqdef\{u\in V\st u\preceq v\},\quad
  \Up{v}\eqdef\{u\in V\st v\preceq u\}\\
\]
Note that $\Down{v}$ is the unique minimal element of $\F$ containing
$v$, and $V\sm \Up{v}$ is the unique maximal element of $\F$ not
containing $v$.
If $D=(V,F)$ is transitive, then for all $v\in V$ $\Down{v}=\{v\}\cup \{u\in V\st  (v,u)\in F \}$  and $\Up{v}=\{v\}\cup \{u\in V\st  (u,v)\in F \}$.

For every $X\subseteq V$, we define
\[\Down{X}\eqdef \bigcup_{v\in X} \Down{v},\quad \Up{X}\eqdef \bigcup_{v\in X} \Up{v}\]
that is, $\Down{X}$ is the unique minimal element of $\F$
containing $X$ and $V\setminus \Up{X}$ is the unique maximal element of $\F$ disjoint from $X$.

\paragraph{The lower bound  and $\lf$}

Let us define the lower bound function
\[
\low(v)\eqdef f((V\setminus \Up{v})\cup \{v\})-f(V\setminus \Up{v}),
\]
as in \cite[(3.95)]{fujishige-book}. Observe that $(V\setminus \Up{v})\cup \{v\}\in {\cal F}$ for every
$v\in V$.
It is easy to see that $z\ge\ell$ for every extreme basis $z\in B(f)$ defined by an ordering
consistent with $\preceq$. The next claim shows that $\ell$ yields lower bounds on
the function values of $f$. In particular,  if $\ell(v)>0$, then no
minimizer of \eqref{prob:sfm} contains $v$.

\begin{claim}\label{cl:ring-low}
For every $X,Y\in{\cal F}$ with $X\subseteq Y$, we have $\ell(Y\setminus X)+f(X)\le f(Y)$. In particular, if we let $P=\{v\in V\st \ell(v)> 0\}$, we have that
$f(X)> f(X\sm \Up{P})$ for all $X\in {\cal F}$.
\end{claim}
\begin{proof}
For the first part of the claim, let us take the elements of $Y\setminus X$ in a consistent order with
$\preceq$ as $z_1,\ldots,z_r$.  Then, $Z_i=X\cup\{z_1,\ldots,z_i\}\in
{\cal F}$ for each $i\in [r]$, and $Z_i\subseteq (V\setminus
\Up{z_i})\cup\{z_i\}$. Submodularity implies $f(Z_i)-f(Z_{i-1})\ge
\ell(z_i)$. The claim follows by adding up all these inequalities.

The second part of the claim follows from first part of the statement if we choose $Y=X\sm \Up{P}$ (note that $X\sm \Up{P}\in{\cal F}$ since it is the intersection of $X$ and $V\sm \Up{P}$, which are both elements of ${\cal F}$).
\end{proof}

Next, we define a submodular function $\lf \st  2^V \to \Z$ as
\[
\lf(X)\eqdef  f(\Down{X})-\low^-(\Down{X} \setminus  X).
\]
As shown in Lemma~\ref{lem:lf} below, this function is submodular,
$\lf(S)=f(S)$ for every $S\in{\cal F}$, and $\lf(S)\ge f(\Down S)$ for
all $S\subseteq V$. Consequently, $\lf$ takes its minimum value on
sets in $\cal F$, and if all minimizers in $f$ are
contained in $\cal F$, then minimizing $\lf$ is equivalent to
minimizing $f$.


An important advantage of using $\lf$ instead of
$f$ is that  $z\ge \ell^-$ for all bases  $z\in B(\lf)$
(Claim~\ref{cl:ell}), and therefore the complexity parameter $L_{\lf}$
is bounded by $|\ell^-(V)|$ (Claim~\ref{cl:L-bound}). This can be arbitrarily
smaller than $L_f$, and therefore {\sc Approx-SFM}$(\lf,\delta)$
returns a higher accuracy solution than  {\sc Approx-SFM}$(f,\delta)$.

We note that the definition of $\lf$ only uses the function values of
$f$ on
$\cal F$. In fact, similar definitions can be used to reduce
submodular function minimization on ring families to minimizing submodular
functions on $2^V$, see e.g. \cite[Section 49.3]{Schrijver03}.

\begin{lemma}\label{lem:lf} The function $\lf$ is submodular
  on $2^V$ with $\lf(S)\ge f(\Down S)$ for all $S\subseteq V$ and
  $\lf(S)=f(S)$ for every $S\in\F$. Consequently, minimizing $f$ on the ring family $\F$ is equivalent
  to minimizing $\lf$ on $2^V$. The complexity of
  \textsc{GreedyMin}$(\lf,w)$ can be bounded by $O(n\cdot \EO+n^2)$,
  where $\EO$ is the complexity to evaluating $f$.
\end{lemma}

\begin{proof}[Proof of Lemma~\ref{lem:lf}]
To prove that $\lf$ is submodular, we show that $\lf=b$ where $b:\, 2^V\rightarrow \Z$ is defined by
\[
b(X)\eqdef\min\{ f(Y)-\low^-(Y\setminus X)\st X\subseteq Y, Y\in {\cal
  F}\}.
\]
Observe that, by Claim~\ref{cl:ring-low}, $\low^-(Y)\le f(Y)$ for every $Y\in {\cal F}$. In particular,
$b$ is the so called {\em contraction of $f$ by vector $\ell^-$}, defined in~\cite[(3.23)]{fujishige-book},
where it is shown  that $b$ is submodular.

Let us show
that $b=\lf$. By definition, $b(X)\le \lf(X)$ for any
$X\subseteq V$. Consider now $Y\in{\cal F}$ such that $Y\supsetneq \Down X$. Again by
Claim~\ref{cl:ring-low}, we have
$f(Y)\ge f(\Down X) +\low^-(Y\setminus \Down X)$; this implies that
$\Down X$ is the minimizer in the definition of $b$, therefore $\lf=b$.

The nonpositivity of $\ell^-$ gives that $\Down f(S)\ge f(\Down S)$
for all $S\subsetneq V$; it is clear that $\Down f(S)=f(S)$ for all
$S\in{\cal F}$, $S\neq V$. Regarding the complexity of
\textsc{GreedyMin}, one needs to compute the values of $\Down
f(\{v_1,\ldots,v_i\})$ for every $i\in[n]$ for a given order of the
vertices; thus, we need to find the sets
$S_i=\Down{\{v_1,\ldots,v_i\}}$. When moving from $i$ to $i+1$, we need to compute
$S_{i+1}=S_i\cup\Down{v_i}$, which can be done in $O(n)$ time because $F$ is maintained as a
transitive digraph. Adding
the $\ell^-$ values also take $O(n)$ time for each set. Hence, we
obtain an overhead $O(n^2)$ over the $O(n \cdot\EO)$ oracle queries and
$O(n \log n)$ time for sorting the ground set.
\end{proof}

\begin{claim}\label{cl:ell}
For every $v\in V$, $\lf(V) - \lf(V\setminus\{v\})\in\{\low^-(v),\low(v)\}$. In particular, $y(v)\ge \low^-(v)$ for every
$y\in B(\lf)$.
\end{claim}
\begin{proof}
If $V\setminus \{v\}\notin{\cal F}$, then $\Down{(V\sm\{v\})}=V$, therefore $\lf(V) - \lf(V\setminus\{v\})=f(V)-f(V)+\low^-(v)$ by definition of $\lf$. If $V\setminus\{v\}\in {\cal F}$,
then $\Up{v}=\{v\}$, therefore $\lf(V) - \lf(V\setminus\{v\})=f(V)-f(V\sm\{v\})=f((V\setminus \Up{v})\cup \{v\})-f(V\setminus \Up{v})=\low(v)$. For the last part, note that for any extreme base $g$ of $B(\lf)$, $g(v)=\lf(S)-\lf(S\sm\{v\})$ for some $S\subseteq V$ containing $v$, and by submodularity $\lf(S)-\lf(S\sm\{v\})\geq \lf(V) - \lf(V\setminus\{v\})\geq\low^-(v)$.
\end{proof}
We remark that the second part of the Claim~\ref{cl:ell} also follows
from~\cite[Theorem 3.6]{fujishige-book}. The next claim shows that
$|\ell^-(V)|$ is a good approximation of the complexity parameter $L_\lf$.
\begin{claim}\label{cl:L-bound}
Assuming that $f(V)\le 0$, we have
$|\low^-(V)|/{|V|}  \le \Lb{\lf}\le 2|\low^-(V)|$.
\end{claim}
\begin{proof}
By Claim~\ref{cl:ell}, $y(v)\ge \low^-(v)$ for every $v\in V$ and
$y\in B(\lf)$,  therefore $y^-(V)\ge \low^-(V)$. Together with
$\lf(V)\le 0$, this implies
$\|y\|_1=\lf(V)-2y^-(V)\le -2\low^-(V)$.

For the lower bound, let us choose $v\in V$ with lowest value of $\low^-(v)$. Thus, $|\low^-(v)|\ge |\low^-(V)|/{|V|}$.
 Consider any extreme base $g$ of $B(\lf)$ from an order where $v$
 comes last. Then by the first part of Claim~\ref{cl:ell} $g(v)\in\{\low^-(v),\low(v)\}$, hence $|g(v)|\geq |\low^-(v)|$, which implies $\Lb{\lf}\ge \|g\|_1\ge
 |g(v)|\geq |\low^-(v)|\ge |\low^-(V)|/{|V|}$.
\end{proof}

\subsection{The basic strongly polynomial scheme}\label{sec:strong-and-simple}
\begin{figure*}[htb!]
\begin{center}
\begin{minipage}{0.85\textwidth}
\begin{algorithm}[H]
\raggedright
  \begin{algorithmic}[1]
    \Require{A submodular function $f:2^V\to \R$  with $f(V)\le 0$, and $\delta>0$.}
    \Ensure{An optimal solution to \eqref{prob:sfm}}.
    \State Initialize $F:=\emptyset$, $T:=\emptyset$.
    \While{$\low^-(V)<0$}\label{l:while}
\State Call \Call{Approx-SFM}{$\lf,\delta$} to obtain $W$ and $y\in
B(\lf)$, represented as a convex combination $y=\sum_{i=1}^k x_i g_i$.
\For{$z\in V$ such that $f(V\setminus \Up{z})> -|V|\cdot y^-(V\setminus \Up{z})$}\label{l:select-z}
\For{$i=1,\ldots,k$},
$g_i'\gets$\Call{Push}{$g_i,\lf,\Up{z}$}.\EndFor
\State $y':=\sum_{i=1}^k x_i g_i'$. \label{l:forward}
\For{$v\in V\setminus \Up{z}$ such that $y'(v)>-{y'}^-(V\setminus \Up{z})$}  \label{l:select-v}
\State add arc $(v,z)$ to $F$. \label{l:add-arc}
\EndFor
\EndFor
\For{$v\in V$ such that {$y(v)<2\low^-(V)\delta$}} \label{l:contract-1} \Comment{contraction}
\State Replace $f$ by $f(S\cup \Down{v})-f(\Down{v})$ on the ground set $V:=V\setminus
\Down{v}$.\label{l:contract-2}
\State Set $f(V):=\min\{0,f(V)\}$.
\State Set $T:= T\cup\Down{v}$.
\EndFor
\State Contract all strongly connected components of $F$ to single
nodes using \eqref{eq:contract}.\label{l:contract-c}

\State Replace $F$ by its transitive closure

\EndWhile
\Return {the pre-image of $T$ in the original ground set.}
 \end{algorithmic}
\caption{The basic strongly polynomial algorithm}\label{alg:strongly-1}
\end{algorithm}
\end{minipage}
\end{center}
\end{figure*}

Algorithm~\ref{alg:strongly-1} builds a ring family $\F$ represented by a
directed graph $F$ with the property that $\F$ contains all optimal
solutions to \eqref{prob:sfm}.

The subroutine {\sc
  Push}$(g,h,S)$ is defined for a submodular function $h:2^V\to \Z$,
an extreme basis $g\in B(h)$ given by an ordering of the elements of $V$, and a
subset $S\subseteq V$. It returns another extreme basis $g'\in B(g)$
by moving the set $S$ backward in the ordering defining $g$. That is,
every element of $V\setminus S$ will precede every element in $ S$, and
the relative ordering of the elements inside $V\setminus S$ as well as those
inside $S$ remain unchanged.

In Theorem~\ref{thm:strongly-1}, we show that for $\delta=1/(3n^3)$, the algorithm terminates within
$n^2$ iterations. Let ${\AO}^{\downarrow}(f,\delta)$ denote the
maximum of $AO(\lf,\delta)$, where $\lf$ ranges over all possible
choices of ring families $\F$.

Recall that Lemma~\ref{lem:lf} provides an upper bound
$O(n^2\cdot\EO+n^2)$ on the running time of
\textsc{Greedy-Min}$(\lf,w)$, in contrast to the running time bound
$O(n^2\cdot\EO+n\log n)$ on \textsc{Greedy-Min}$(f,w)$.
Thus, $\AO^\downarrow(f,\delta)$ can be upper
bounded by the worst case running time bound on $\AO(f,\delta)$ plus
$O(n/\log n)$ times the worst case bound on the number of calls to
the greedy algorithm in $\AO(f,\delta)$.

\begin{theorem}\label{thm:strongly-1}
Using $\delta=1/(3n^3)$,
 Algorithm~\ref{alg:strongly-1} finds the unique minimal optimal
  solution to \eqref{prob:sfm} in time $O(n^2 (\AO^\downarrow(f,1/(3n^3))+n^4
  \cdot\EO+n^5)$.
\end{theorem}
Using the bounds from Theorems~\ref{thm:cond-grad} and \ref{thm:wolfe},
we obtain $O(n^{10}\cdot \EO+n^{11})$ using the conditional gradient
algorithm, and $O(n^{11}\cdot \EO+n^{12})$ using the Fujishige-Wolfe
algorithm.
While these running times are high degree polynomials, we
emphasize that they can be obtained by repeated applications of simple
iterative methods, without using any form of scaling.

Theorem~\ref{thm:oracle-running
  time} gives a running time  $O((n^6 \cdot \EO+n^7)\log n)$ using the
\textsc{Rescaling-SFM} algorithm. In Section~\ref{sec:strong-fast}, we
give an enhanced version of the algorithm with running time  $O((n^5\cdot
\EO+n^6)\log^2 n)$.

\medskip

Let us now give an overview of
Algorithm~\ref{alg:strongly-1}. The ground set is reduced by
contracting elements that must be included in every optimal solution.  The set
$T$ represents the current set of contracted elements. Thus, the submodular function at the current stage will be
defined as $f(S\cup T)-f(T)$ for the original input function $f$, with the
possible exception of $f(V)$. Therefore, the complexity of evaluating the
current $f$ is still $\EO$.  We will use $n$ below to denote the size of the
original ground set $V$.

Once $\ell(v)\ge 0$ for
  all $v\in V$, we conclude that $S=\emptyset$ is the minimizer of the
  current function; we return $T$ as the unique minimal minimizer of the original function.
Otherwise, as long as $\ell^-(V)<0$, each
main iteration calls the oracle
\textsc{Approx-SFM}($\lf,\delta)$.
Two types of contractions are used.
All cycles in $F$ can be contracted to
single elements using the operation \eqref{eq:contract}, since an optimal solution can contain either all or
no element of a cycle (line~\ref{l:contract-c}). The other type of contraction (in
line~\ref{l:contract-2}) reduces the size of
the ground set by eliminating
elements that must be contained in every optimal solution.

The other main step of the algorithm is
adding new arcs to $F$.
The following lemma shows the validity of these
steps and that either of these operations should occur in every
iteration.

\begin{lemma}\label{lem:contract-kosher} The following properties are
  valid throughout Algorithm~\ref{alg:strongly-1}.
\begin{enumerate}[(i)]
\item All minimizers to \eqref{prob:sfm} are contained in the ring
  family $\cal F$.\label{part:ring-fam}
\item  If $\ell^-(V)=0$ then $f(Y)\ge 0$ for all $Y\subseteq V$. \label{part:nonneg}
\item Every arc $(v,z)$ added to
$F$ in line~\ref{l:add-arc} satisfies the property that every
minimizer that contains $v$ must also contain $z$. \label{part:arc}
\item Every $v\in V$ contracted in line~\ref{l:contract-2} must be contained
in all minimizers of \eqref{prob:sfm}.\label{part:contract}
\item If $\delta\le
1/(3n^3)$, then every iteration
either contracts an element or adds a new arc to $F$.\label{part:progress}
\end{enumerate}
\end{lemma}
\begin{proof}
We show all properties by induction on the number of iterations of the
algorithm. Part {\em \eqref{part:ring-fam}} is trivially true at the
beginning, since $\cal F=\emptyset$. Part \eqref{part:arc} guarantees
that it remains true whenever $F$ is extended (and thus $\cal F$ is
restricted to a smaller ring family).

\smallskip

For \emph{\eqref{part:nonneg}}, Claim~\ref{cl:ring-low} shows $f(Y)\ge
\ell(Y)\ge \ell^-(V)$ for any $Y\in {\cal F}$. Our assumption is $\ell^-(V)=0$, that
is $\ell\ge 0$, and thus $f(Y)\ge 0$ for all $Y\in {\cal F}$. By
{\em \eqref{part:ring-fam}}, it follows that $f(Y)\ge 0$ on every set
$S\subseteq V$.
For the rest of the proof, we
assume that $\ell^-(V)<0$.

\smallskip

For {\em\eqref{part:arc}},  Lemma~\ref{lem:extend-U} implies that for $v$
selected in line~\ref{l:select-v}, every minimizer of $\lf$ that
contains $v$, must also contain some element of $\Up{z}$. Using
{\em\eqref{part:ring-fam}}, and the definition of $\lf$, the minimizers of
$f$ are also minimizers of $\lf$, and are contained in $\cal F$. Thus,
if $S$ minimizes $f(S)$ and $S\cap \Up{z}\neq\emptyset$, then
$z\in S$. It follows that every minimizer of $f$ containing $v$ must also contain $z$, therefore the new arc $(v,z)$ is valid.

\smallskip

For {\em\eqref{part:contract}}, consider a $v$ such that
$y(v)<2\low^-(V)\delta$ as in
line~\ref{l:contract-2}. Lemma~\ref{lem:contract-element} and
Claim~\ref{cl:L-bound} imply that $v$ is contained in every minimizer
of $\lf$. By {\em \eqref{part:ring-fam}}, $\Down{v}$ must be contained in all elements of $\Down{v}$.

\smallskip

We now turn to the proof of {\em\eqref{part:progress}}. Note that, if the algorithm enters the while loop when $|V|=1$, say $V=\{v\}$, then $y(v)=f(V)=\ell(\{v\})=\ell^-(V)<2\ell^-(V)\delta$, so the algorithm contracts
$v$ in line~\ref{l:contract-2}, and subsequently terminates. Thus, we
can assume that $|V|\geq 2$. The argument is as follows: we note that
if no element is contracted in line~\ref{l:contract-2}, then the
condition in line~\ref{l:contract-1} yields
\begin{equation}\label{eq:y-bound}
y(v)\geq 2\low^-(V)\delta\quad \forall v\in V.
\end{equation}
From this property, we will derive that the condition in
line~\ref{l:select-z} is true for at least one $z\in V$. Finally, we
will show that in this case, the condition in line~\ref{l:select-v}
is true for some $v\in V\setminus \Up{z}$.

In more detail, if no element is contracted in line~\ref{l:contract-2}, then
from \eqref{eq:y-bound} we get the lower bound
\begin{equation}\label{eq:lb-y}
\lf(S)\geq y(S)\ge y^-(V)\geq 2{|V|}\cdot\low^-(V)\delta\quad \forall S\subseteq V.
\end{equation}
Since $\lf(S)=f(S)$ for $S\in\F$ by the definition of $\lf$, and $\F$
contains all minimizers of $f$ by \eqref{part:ring-fam}, we have that
\begin{equation}\label{eq:f-lb}
f(S)\ge 2{|V|}\cdot \low^-(V)\delta \quad \forall S\subseteq V.
\end{equation}

Next, we show that if \eqref{eq:lb-y}  holds, then at least one $z\in V$ satisfies
\begin{equation}\label{eq:z-req}
f(V\setminus \Up{z})> {|V|} \cdot |y^-(V)|,
\end{equation}
a bound which is slightly stronger than  the condition $f(V\setminus \Up{z})> -{|V|}\cdot y^-(V\setminus \Up{z})$ in
line~\ref{l:select-z}. 
Select $z\in V$ such that $\ell(z)$ is the most negative
possible. In particular, $\ell(z)\leq\ell^-(V)/{|V|}$. By \eqref{eq:f-lb},
we have
\begin{equation}\label{eq:U-z-lower}
\frac{\ell^-(V)}{|V|}\geq\low(z) = f((V\setminus \Up{z})\cup\{z\})-f(V\setminus
\Up{z})\ge 2{|V|}\cdot\low^-(V)\delta-f(V\setminus \Up{z}).
\end{equation}

Consequently,
\[
f(V\setminus \Up{z})\ge {|V|}\cdot |\ell^-(V)|\cdot\left(\frac{1}{{|V|}^2}-2\delta\right).
\]
From  the assumption $\delta\leq 1/(3n^3)\leq 1/(3|V|^3)$ we obtain $1/{|V|}^2-2\delta>
2{|V|}\delta$ since ${|V|}\ge 2$. Therefore \eqref{eq:z-req} follows since
\[
f(V\setminus \Up{z})> 2 |V|^2 \cdot |\ell^-(V)|\delta \ge {|V|}\cdot |y^-(V)|.
\]
In the final inequality we used \eqref{eq:lb-y} (recalling that $\ell^-(V)$ and $y^-(V)0$ are non-positive).

It is left to show that if $z\in V$ satisfies the condition in
line~\ref{l:select-z}, then at least one  $v\in V\setminus \Up{z}$
will be selected in line~\ref{l:select-v}.
By construction and from the fact that $V\setminus \Up{z}\in {\cal F}$, for every $i\in [k]$ we have $g'_i(V\setminus \Up{z}) =\lf(V\setminus \Up{z})=f(V\setminus \Up{z})$, and $g'_{i}(u)\ge g_{i}(u)$ for every $u\in V\setminus \Up{z}$. It follows that $y'(V\setminus \Up{z})=f(V\setminus \Up{z})$ and $y'(u)\ge y(u)$ for all $u\in V\setminus \Up{z}$.

Assume that $f(V\setminus \Up{z})> -{|V|}\cdot y^-(V\setminus \Up{z})$, as in the
condition in line~\ref{l:select-z}.
It follows that
\begin{eqnarray*}
y'(V\setminus \Up{z})=f(V\setminus \Up{z}) > -{|V|}\cdot y^-(V\setminus \Up{z})\ge -{|V|}\cdot {y'}^-(V\setminus  \Up{z}).
\end{eqnarray*}
This in turn implies the existence of  $v\in V\setminus \Up{z}$ such that $y'(v)>- {y'}^-(V\setminus  \Up{z})$
in line~\ref{l:select-v}.

\end{proof}

\begin{proof}[Proof of Theorem~\ref{thm:strongly-1}]
Lemma~\ref{lem:contract-kosher} justifies the contraction steps and
the addition of new arcs to $F$, and shows that the number of main
iterations is at most $n^2$. Let us also note that after every
contraction, we decrease the value of $f(V)$ if it becomes positive
(that is, if $f(V)>f(\Up{v})$ before the contraction of $v$). This
operation clearly maintains submodularity. It is also safe in the
sense that it may not lead to an incorrect output with respect to the
original function. Indeed, note that at termination the algorithm
returns the current set of $T$, which are elements that must be
contained in every minimizer of the original function. Hence, the
algorithm outputs the unique minimal solution to \eqref{prob:sfm}. On
the other hand, if $f(V)$ was ever decreased, then we decrease it to the
same value as $f(\emptyset)$. Therefore it can never become the unique
minimizer. If the algorithm terminates with the entire ground set $V$,
then it follows that $f(V)$ was never decreased during the algorithm.

Let us now estimate the running time. Besides the calls to
\textsc{Approx-SFM}, the running time is dominated by the operations \textsc{Push}$(g_i',\lf,\Up{z})$, which altogether require $O(n^2\cdot \EO
+ n^3)$ for every iteration, and this is required $O(n^2)$ times.
Every time an arc $(u,v)$ is added to $F$, recomputing the transitive
closure requires to add arcs from $u$ and all its predecessors to $v$ and
all its descendants. This requires $O(n^2)$ operations per added arc, so
it requires $O(n^4)$ operations overall, which is within the stated running time bound.
\end{proof}

\subsection{Speeding up the algorithm}\label{sec:strong-fast}
The algorithm described in the previous section needs to identify
$O(n^2)$ arcs in $F$. In the worst case, each iteration may only
identify a single arc, resulting in $O(n^2)$ calls
to {\sc Approx-SFM}.

On the other hand, if  we were able to guarantee that $|\low^-(z)|$ is within a  factor
$O(n^\bk)$ from $|\low^-(V)|$ for a constant fraction of all $z\in V$
for some constant $\bk\ge 1$, the analysis in the proof of Lemma~\ref{lem:contract-kosher} implies that for $\delta=1/O(n^{\bk+2})$ we would guarantee $f(V\sm \Up{z})\geq -n y^-(V\sm \Up{z})$ for all such $z\in V$.
Thus, after running  {\sc Approx-SFM}$(\lf, 1/O(n^{\bk+2}))$, we could
extend $F$ by $\Theta(n)$ new arcs.

If this property held in all iterations, then $O(n)$ calls to
{\sc Approx-SFM} would suffice. However,
the number of $z\in V$ with $|\low^-(z)|$ value ``close'' to $|\low^-(V)|$
can be $o(n)$. To deal with this situation, we apply the \emph{bucketing}
technique of Lee, Sidford, and Wong \cite{LSW}. Instead of the entire
$V$, we restrict our function in every iteration to a suitably chosen $\bar V\subseteq V$, and run
{\sc Approx-SFM} restricted to this set with $\delta=n^{-O(\log n)}$. We will obtain $\Theta(|\bar V|)$ new
arcs in this iteration. Thus, if {\sc Approx-SFM} has running time
$O((|\bar V|^4\cdot\EO+|\bar V|^5)\log^2 n)$, then the amortized cost of
extending $F$ by an arc will be $O((n^3\cdot\EO+n^4)\log^2 n)$.

We note that this improvement is only applicable if
$\AO(f,\delta)$ depends logarithmically on $1/\delta$. Since $\delta$
can be quasi-polynomial, the conditional gradient or Fujishige-Wolfe
methods would not even be polynomial in this framework.

Another speedup technique, also adapted from \cite{LSW}, enables to
save on the running time of recomputing the extreme bases in
line~\ref{l:forward} of Algorithm~\ref{alg:strongly-1}. We can identify
the new arc $(z,v)$ by recomputing only one of the $g_i'$s instead of
all of them, at the expense of requiring a higher accuracy
$2\k\delta$ instead of $\delta$.  Here, $\k$ denotes an upper bound on
the number of extreme bases in the convex combination, hence $\k\in O(n)$ by assumption.

The following lemma adapts the argument in
Section 15.4.1 in \cite{LSW}.

\begin{lemma}\label{lem:bucketing} Let  $f:2^V\to\mathbb{Z}$ be a submodular function,
$\cal F$  a ring family containing all minimizers of $f$, and  $\lf:2^V\to\mathbb{Z}$ be the corresponding function defined by $f$ and $\cal F$.
Then in $O(n \cdot \EO)$ time we can find a nonempty subset $\bar V\subseteq V$ and a positive integer $\bk=O(\log n)$,
such that
\begin{itemize}
\item For every $z\in V\setminus \bar V$, we have $\low(z)>
  2\low^-(V)/(4\kappa n)^{4\bk}$.
\item There exist at least $|\bar V|/2$ distinct $z\in \bar V$ such
  that $\low(z)\le 2\low^-(V)/(4\kappa n)^{4\bk-4}$.
\end{itemize}
\end{lemma}
\begin{proof}
Let us define
$
V^t\eqdef\{z\in V: \low(z)\le 2\low^-(V)/(4\kappa  n)^{4t})\}$ for  $t=1,2,\ldots
$. Note that $V^1\subseteq V^2\subseteq \ldots$, because $\ell^-(V)\le 0$. Furthermore, $V^1\neq\emptyset$, as it contains $z$ with the smallest
$\low(z)$ value. Let $\bk$ be the smallest value such that $|V^\bk|\le
2|V^{\bk-1}|$. Thus, $\bk=O(\log  n)$, and choosing $\bar V=V^b$ satisfies
both requirements.
\end{proof}

For the set $\bar V$ and value $\bk$ as in the lemma, let $\bar f:2^{\bar
V}\to \mathbb{Z}$ denote the restriction of $\lf$ to the ground set
$\bar V$, and let us set
\begin{equation}\label{eq:barepsilon}
\bar\delta\eqdef\frac{1}{(4 \k  n)^{4\bk}},\quad \delta\eqdef
\frac{2 n^2+1}{(4\k  n)^{4\bk}},
\end{equation}

Let us call
{\sc Approx-SFM}$(\bar f,\bar\delta)$ to obtain the vector $\bar
y\in B(\bar f)$ defined as a convex combination of extreme bases $\bar
g_1,\ldots,\bar g_k\in B(\bar f)$ with $k\le\k$, and a set $W\subseteq \bar V$ such that $\bar f(W)\le \bar
y^-(\bar V)+\bar\delta \Lb{\bar f}$.

Let us now extend $\bar y\in\R^{\bar V}$ to $y\in\R^V$ as follows. For $v\in \bar V$, we let $y(v)=\bar
y(v)$. Then, consider an arbitrary order $v_1,\ldots,v_{ n-|\bar V|}$ of
$V\setminus \bar V$, and set $y({v_j}):=\Down{f}(\bar V\cup\{v_1,\ldots,v_{j}\})-\Down{f}(\bar V\cup\{v_1,\ldots,v_{j-1}\})$. Let us also define $g_1,\ldots,g_k\in \R^V$, by
$g_i(v)=\bar g_{i}(v)$ for $v\in \bar V$, $g_i(v)=y(v)$ for $v\in V\sm \bar V$ ($i=1,\ldots,k$). Note that, by definition, $g_1,\ldots,g_k$ are extreme bases of $B(\lf)$, and $y$ is a convex combination of $g_1,\ldots,g_k$.

\begin{lemma}
For the vector $y$ and set $W$ as above, we have that $y\in B(\lf)$, and  $\lf(W)\le
y^-(V)+\delta\Lb{\lf}$.
\end{lemma}
\begin{proof}
By definition
$\lf(W)=\bar f(W)$ and $\Lb{\bar f}\le \Lb{\lf}$, because $\bar
f$ is a restriction of $\lf$. Therefore,
\[
\lf(W)\le \bar y^-(\bar V)+ \bar\delta \Lb{\lf}.
\]
\begin{claim}
$y(v)\ge \low^-(v)$ for every $v\in V\setminus \bar
V$.
\end{claim}
\begin{proof}
If $v=v_j$, then  $y(v)=\Down{f}(\bar
V\cup\{v_1,\ldots,v_{j}\})-\Down{f}(\bar
V\cup\{v_1,\ldots,v_{j-1}\})\ge \lf(V)-\lf(V\setminus \{v\})$ by
submodularity.
Further, $\lf(V)-\lf(V\setminus \{v\})\ge \low^-(v)$ by Claim~\ref{cl:ell}.
\end{proof}
We have $y^-(V)=\bar y^-(\bar V)+y^-(V\setminus
\bar V)$. By the choice of $\bar V$,
we have
 \[
\low^-(v)\ge 2\low^-(V)/(4\kappa  n)^{4b}=2\ell^-(V)\bar\delta\quad
\forall v\in
V\setminus \bar V.
\]
 Using the claim above,
we get $y^-(V\setminus \bar V)\ge
2 n \ell^-(V)\bar\delta$. Thus, $y^-(V)\ge \bar y^-(\bar
V)+2 n \ell^-(V)\bar\delta\ge \bar y^-(\bar
V)- 2 n^2\bar\delta\Lb{\lf} $.  Here, the last inequality used the
lower bound in
Claim~\ref{cl:L-bound}.
Consequently,
\[
\lf(W)\le y^-(V)+(2 n^2+1)\bar\delta \Lb{\lf}=y^-(V)+\delta \Lb{\lf}.
\]
\end{proof}
This proof shows that we can implement {\sc Approx-SFM}$(f,\delta)$
by calling {\sc Approx-SFM}$(\bar
f,\bar\delta)$, and adding the remaining $V\setminus \bar V$
elements by $O( n)$ value oracle queries for $\lf$, and $O(n\k)$ arithmetic
operations.  This  gives a running time overhead
$O( n\cdot\EO+ n^2)$.

We make two modifications to Algorithm~\ref{alg:strongly-1} as
follows. Firstly, in every
iteration, we compute $\bar V$ and $\bk$ as in
Lemma~\ref{lem:bucketing}, and use this modified implementation of
{\sc Approx-SFM}
with
$\delta$ as defined in \eqref{eq:barepsilon}.

Secondly, we modify the selection of $z$ as in line~\ref{l:select-z}
to
\begin{equation}\label{eq:half-eq}
f(V\sm \Up{z})> -2k|V|\cdot y^-(V\sm \Up{z}).
\end{equation}
In line~\ref{l:forward}, we
change the computation of $y'$ as
follows. We first compute $x_i ( f(V\sm\Up{z}) -g_i(V\sm\Up{z}))$ for
all $i\in [k]$; w.l.o.g., assume the maximum is taken for $i=1$. We
compute $g_1'\gets$\Call{Push}{$g_1,\lf,\Up{z}$} and define
\[
y':=y+x_1(g_1'-g_1).
\]
The rest of the algorithm remains unchanged.

\begin{theorem}\label{thm:strongly-2}
The above described modification of Algorithm~\ref{alg:strongly-1}
finds an optimal solution to \eqref{prob:sfm} in time
$O( n \cdot\AO^\downarrow(f, n^{-O(\log  n)})+ n^3 \cdot
  \EO+ n^4 )$. Using the implementation with
\textsc{Rescaling-SFM}, the running time is $O(( n^5\cdot \EO+ n^6)\log^2  n)$.
\end{theorem}

\begin{proof}
Note that properties (\ref{part:ring-fam})-(\ref{part:contract}) of Lemma~\ref{lem:contract-kosher} are maintained throughout the algorithm, just as before.
We need to prove the following analogue of Lemma~\ref{lem:contract-kosher}(\ref{part:progress}).
\begin{claim}\label{cl:half}
After every call of the approximation oracle, either at least one node
is contracted at line~\ref{l:contract-2}, or at least
$\frac12|\bar V|$ new arcs are added to $F$.
\end{claim}
\begin{proof}
We will first show that, if no node is contracted  at
line~\ref{l:contract-2}, then \eqref{eq:half-eq}
holds for at least half of the elements $z$ of
$\bar V$. Subsequently, we will show that, for every $z$ satisfying \eqref{eq:half-eq},
there exists $v\in V\sm \Up{z}$ satisfying the condition in line \ref{l:select-v}.

If no node is contracted  at line~\ref{l:contract-2}, then as in the proof of Lemma~\ref{lem:contract-kosher}(\ref{part:progress}),
we can assume that $y(v)\ge 2\low^-(V)\delta$ for all $v\in V$;
further, \eqref{eq:lb-y} and \eqref{eq:f-lb} hold.
By Lemma~\ref{lem:bucketing} and our choice of $\bar V$,  at least half of
the elements $z$ of $\bar V$ satisfy
\[\low(z)\le
2\low^-(V)/(4\k n)^{4\bk-4}=2\low^-(V)\delta (4\k n)^4/(2n^2+1).\]
Consider any such $z$.
As in \eqref{eq:U-z-lower}, the assumption \eqref{eq:f-lb}
implies that
$$f(V\setminus \Up{z})\geq 2|V|\cdot \low^-(V)\delta-\ell(z)
\geq
2|V|^2\cdot |\low^-(V)|\delta\left(\frac{(4\k
     n)^4}{|V|^2(2n^2+1)}-\frac{1}{|V|}\right).$$
Using \eqref{eq:lb-y},  $2|V|\cdot|\low^-(V)|\delta\ge |y^-(V)|$, and
it is easy to see that the expression in the brackets is $>2\kappa\ge
2k$ for $n\ge 2$. Thus,
\[
f(V\setminus \Up{z})\ge 2k|V|\cdot|y^-(V)|,
\]
implying \eqref{eq:half-eq}.

Let us now show that for any $z$ satisfying \eqref{eq:half-eq}, there
exists a $v\in V\setminus \Up{z}$ such that $y'(v)>-{y'}^-(V\setminus
z)$.  Recall our assumption that $i=1$
maximizes $x_i ( f(V\sm\Up{z}) -g_i(V\sm\Up{z}))$, and therefore
\[
x_1 ( f(V\sm\Up{z}) -g_1(V\sm\Up{z}))\ge \frac{1}{k} \left(f(V\sm\Up{z})
-y(V\sm\Up{z})\right).
\]
Also note that $g_1'$ is defined so that $g_1'(V\setminus \Up{z})=\lf
(V\setminus \Up{z})=f(V\setminus \Up{z})$. We get

\[
\begin{aligned}
y'(V\setminus \Up{z})&=y(V\setminus \Up{z})+ x_1 ( g_1'(V\sm\Up{z})
  -g_1(V\sm\Up{z})) \\
&=y(V\setminus \Up{z})+ x_1 ( f(V\sm\Up{z})
  -g_1(V\sm\Up{z}))\\
& \ge\frac{k-1}{k} y(V\setminus \Up{z}) +
  \frac{1}{k} f(V\sm\Up{z})\\
\end{aligned}
\]
Using \eqref{eq:half-eq}, we obtain
\[
y'(V\setminus \Up{z}) > -{|V|}\cdot y^-(V\setminus \Up{z})\ge -{|V|}\cdot {y'}^-(V\setminus  \Up{z}),
\]
guaranteeing the existence of  $v\in V\setminus \Up{z}$ such that $y'(v)>- {y'}^-(V\setminus  \Up{z})$
in line~\ref{l:select-v}.
\end{proof}

The running time of
\textsc{Approx-SFM}$(\bar f,\bar \delta)$ is
$\AO(\bar f,\bar\delta)$. There are at most $n$ iterations where a
node gets contracted; the total cost of the oracle calls in these
iterations can be bounded by $n\AO^\downarrow(f,\bar\delta)$.

Consider now the iterations when no nodes get contracted.
In these iterations, the amortized cost of an oracle call per new arc is
$2\AO(\bar f,\bar\delta)/|\bar V|$. Since $\AO$ depends at least linearly on $|\bar V|$, this can be upper
bounded by $\AO^\downarrow(f,\bar\delta)/|V|$.
Hence, the total time of the oracle calls is $O( n\cdot
\AO(f,\bar\delta))$.

After every call, there is an overhead $O(n\cdot \EO+n^2)$,
totalling $O(n^3\cdot \EO+n^4)$. For every arc
identified, it takes $O(\EO+\kappa)$ to identify which $g'_i$
needs to be computed, and it takes $O( n\cdot\EO+n^2)$ time to compute
this $g_i'$; this takes $O(n^3 \cdot \EO+n^4)$ time overall. As in the proof
of Theorem~\ref{thm:strongly-1}, recomputing the transitive closures requires $O(n^4)$ operations
over the entire execution of the algorithm.

Finally, if we consider \textsc{Approx-SFM} provided by algorithm \textsc{Rescaling-SFM} as in  Theorem~\ref{thm:oracle-running
  time}, we obtain a running time bound
$O(( n^5\cdot \EO+ n^6)\log^2 n)$.
\end{proof}


\section{Cutting plane method}\label{sec:ellipsoid}
The current best cutting plane method for finding a point in a convex
set $C\subseteq \R^n$ provided by a separation oracle is due to Lee, Sidford, and Wong
\cite{LSW}. Assume that, for  $R>\varepsilon>0$, $C$ is contained in a ball of radius $R$ centered at the origin and contains some ball of radius $\varepsilon$. Let $\kappa:=nR/\varepsilon$.
In general, cutting plane methods maintain at each iteration $k$ a ``simple'' convex set $K^{(k)}$ such that $C\subseteq K^{(k)}$, and select a candidate point $x^{(k)}\in K^{(k)}$. With a call to the separation oracle for $C$, the method either terminates if $x^{(k)}\in \intr(C)$, or else it generates a valid inequality for $C$ that weakly separates $x^{(k)}$ from $\intr(C)$, and uses this inequality to generate a new relaxation $K^{(k+1)}$ with smaller volume and a new candidate point $x^{(k+1)}$.

Lee, Sidford, and Wong's method~\cite{LSW} maintains $K^{(k)}$ as the intersection of $O(n)$ valid inequalities, and guarantees that the volume of $K^{(k)}$ decreases by a constant factor at every separation oracle call. This ensures that the number of oracle calls is bounded by $O(n\log\kappa)$. The overall running time is  $O(n\cdot \mathrm{SO}\log\kappa+n^3\log^{O(1)} \kappa)$, where $\mathrm{SO}$ is the
complexity of an (exact) separation oracle.

By contrast, the central-cut ellipsoid method maintains $K^{(k)}$ to be an ellipsoid, where the volume of $K^{(k)}$ decreases by a constant factor every $n$ separation oracle calls (see \cite{glsbook}, Lemma~3.2.10), ensuring a $O(n^2\log\kappa)$ bound on the number of oracle calls. The overall running time is  $O((n^2\cdot \mathrm{SO}+n^4)\log\kappa)$\footnote{While not explicitly stated in~\cite{glsbook}, this running time bound follows immediately from the analysis.}

In Part III~\cite{LSW}, Lee, Sidford, and Wong  apply their cutting plane method to submodular
function minimization, and obtain a strongly polynomial running time bound of
$O(n^3 \log^2 n \cdot \EO +n^4 \log^{O(1)} n)$, which is currently the best (see \cite[Section 15.4]{LSW}). This is obtained by combining their
cutting plane algorithm with an improved version of the combinatorial framework of ring
families; one of their important new contributions is the bucketing technique we
also use in
Section~\ref{sec:strongly}.

In this section, we present an alternative way of applying their
cutting plane method to SFM. We prove the same running time bound in a
substantially simplified way. Firstly, instead of using the Lov\'asz
extension as in \cite{glsbook} and in \cite{LSW}, we apply the cutting plane
method to find a feasible solution in $F_\mu$, as defined in
\eqref{eq:sigma-def}. We use the sliding technique as in Section~\ref{sec:weakly} for the
cutting plane algorithm. Secondly, we employ the combinatorial framework
in a black-box manner, by implementing \textsc{Approx-SFM} via the
Lee-Sidford-Wong algorithm. The combinatorial interpretation of the certificate returned
by the cutting plane method turns out to be much easier than in \cite{LSW}.

\paragraph{Weakly polynomial algorithm}
Let us start by exhibiting a weakly polynomial  $O(n^2 \log
(n\Lbt{f}) \cdot \EO +n^3 \log^{O(1)} (n\Lbt{f}))$ algorithm for
SFM, which is the same as the running time in \cite{LSW}. We use a
slight modification of the
cutting plane algorithm \cite[Section 6.4,
Algorithm 2]{LSW}.

We start with $\mu=\max\{0,-f(V)\}$, and maintain a set $W\subseteq V$ with $f(W)=-\mu$
throughout. The algorithm seeks a point in $\intr(F_\mu)$, and the initial relaxation $K^{(0)}$ is
the hypercube centered at the origin of side length $2\sqrt{n}$. For the current iterate $x^{(k)}$,
\textsc{GreedyMin}$(f_\mu,x^{(k)})$ is used as the separation oracle
for $\textrm{int}(F_\mu)$, which
returns an extreme base $g$ of $B(f_\mu)$. If $g^\T x^{(k)}>0$, then
$x^{(k)}\in \textrm{int}(F_\mu)$, thus $x^{(k)}$ is feasible. In this case,
instead of terminating, we modify the value of $\mu$ as in the sliding
von Neumann algorithm.  That is, we set
$W=\textsc{MinSet}(f_\mu,x^{(k)})$, and set the new value $\mu'=-f(W)$.
From Lemma~\ref{lem:minset}, we see that $x^{(k)}\notin
\textrm{int}(F_{\mu'})$. Thus, we can continue with adding a new
cutting plane.
Note that $F_{\mu'}\subseteq F_{\mu}$ if $\mu'>\mu$, hence the current relaxation $K^{(k)}$ remains valid,
because $F_{\mu'}\subseteq F_{\mu}\subseteq K^{(k)}$.
(Again, this is similar to the
sliding objective technique, although we are changing all constraints
of the polytope simultaneously.) When $-\mu$ is the minimum value of $f$, $L_\mu$ has no points in the interior, therefore we stop when the volume of the current relaxation becomes too small.

In this setting, we have $\mathrm{SO}=n\cdot\EO+n\log n$. For every value of $\mu$,
$F_\mu\subseteq \mathbb{B}^n$ by definition, and Lemma~\ref{lem:omega-bound}
implies that, as long as $\min_{S\subseteq V} f(S)< -\mu$,
$F_{\mu}$ contains a ball of radius $1/(4{n}\Lbt{f})$. Hence,
$\kappa = O({n}\Lbt{f})$, giving the desired running time bound.

Let us  note that the framework just described does not depend on the
specifics of how the relaxations $K^{(k)}$ or the candidate point $x^{(k)}$ are constructed, but
it can be applied to any cutting plane method. For example, for the central-cut ellipsoid method, the above framework
gives an $O((n^3\cdot \EO +\ n^4)\log (n\Lbt{f}))$ algorithm for submodular function
minimization. Indeed, as previously mentioned, the running time of the ellipsoid method is $O((n^2\cdot \mathrm{SO}+n^4)\log \kappa)$, and as before $\mathrm{SO}=n\cdot\EO+n\log n$ and $\kappa = O({n}\Lbt{f})$.
Interestingly, even such a simple and direct use of the standard ellipsoid method, compared to the usual approach of minimizing the Lov\'asz extension, provides a running time that is a factor $n$ lower than any weakly-polynomial SFM-algorithm known prior to the work of Lee-Sidford-Wong~\cite{LSW}.

\paragraph{Strongly polynomial algorithm}
Let us now show an $O(n^2\log(n\delta^{-1}) \EO+n^3\log^{O(1)}(n\delta^{-1}))$
implementation of  \textsc{Approx-SFM}$(f,\delta)$ using the
Lee-Sidford-Wong cutting plane method.
We use Theorem 31 from \cite{LSW}. For $K=F_\mu$ (for any value
of $\mu$), by definition $F_\mu\subseteq \mathbb{B}^n\subseteq
\mathbb{B}^n_\infty(1)$, that is, $R=1$.
Recall, as described above, that we slide $\mu$ every time we find a feasible solution in $F_\mu$.
The following lemma shows that the algorithm  always returns a thin direction as follows.

\begin{theorem}[{\cite[Theorem 31]{LSW}}]
For~any~$\varepsilon\in[0,1]$, in expected time $O(n
\log(n/\varepsilon)\cdot \mathrm{SO}+n^3\log^{O(1)}(n/\varepsilon))$, the
(sliding) cutting plane method returns a value $\mu$, and
constraints $a_i^\T x\ge b_i$ for $i\in [h]$, where $h=O(n)$,
$\|a_i\|_2=1$, which are all valid for $F_\mu$.
 Each of these constraint is either an original box
constraint, that is $x_j\ge -1$ or $-x_j\ge -1$, or an inequality
returned by the separation oracle. Let $P$ denote the intersection of
these hyperplanes.

Further, we obtain non-negative numbers $t_1,t_2,t_3,\ldots,t_h$ with
$t_1=1$, and a point $x^*\in P$, which satisfy the following:

\begin{enumerate}[(a)]
\item $\|x^*\|_2\le 3\sqrt{n}$,
\item $\left\|\sum_{i=1}^h t_i a_i\right\|_2=O(\sqrt{n}\varepsilon \log(1/\varepsilon))$,
\item $a_1^\T x^*- b_1\le \varepsilon$,
\item $ \left(\sum_{i=2}^h t_i a_i\right)^\T x^*-\sum_{i=2}^h t_i b_i\le O(\sqrt{n}\varepsilon\log(1/\varepsilon))$.
\end{enumerate}
\end{theorem}
The output certifies that the region $P\cap \mathbb{B}^n_\infty(1)$ has
small width in the direction of $a_1$. Indeed, let $\bar
a:=\sum_{i=2}^h t_i a_i$ and $\bar b:=\sum_{i=2}^h t_i b_i$. By Cauchy-Schwartz and (b), for all $x\in\R^n$, $|(a_1+\bar a)^\T x|\leq \|x\| O(\sqrt{n}\varepsilon \log(1/\varepsilon))$, so  $b_1+\bar b\leq (a_1+\bar a)^\T x^*\leq O(n\varepsilon \log(1/\varepsilon))$. By (c) and (d), $-b_1-\bar b\leq -(a_1+\bar a)^\T x^*+\varepsilon+O(\sqrt{n}\varepsilon \log(1/\varepsilon))\leq O(n\varepsilon \log(1/\varepsilon))$. This shows $|b_1+\bar b|=O(n\varepsilon \log(1/\varepsilon))$.  It follows that, for every $x\in P\cap \mathbb{B}^n_\infty(1)$,
\[b_1\leq a_1^\T x\leq a_1^\T x+\bar a^\T x-\bar b=(a_1+\bar a^\T) x-\bar b-b_1+ b_1\leq O(n\varepsilon\log(1/\varepsilon))+b_1.
\]

We show that for an appropriately chosen $\varepsilon$, this can be
used to implement \textsc{Approx-SFM}$(f,\delta)$.

\begin{lemma}\label{lem:convert}
For an appropriate $\varepsilon$ such that
$\delta = \Omega(n^{3/2}\varepsilon \log(1/\varepsilon))$, from the output
of the cutting plane method we can obtain $W$
and $y$ as required for \textsc{Approx-SFM}$(f,\delta)$, that is,
$f(W)\le y^-(V)+\delta\Lb{f}$.
\end{lemma}
\begin{proof}
Let $[h]=I_b\cup I_s$, where $I_b$ is the set of indices $i$ such that
$a_i^\T x_i\ge b_i$ is a box constraint, and $I_s$ is the set of
indices corresponding to constraints from the separation oracle. Each
constraint in $I_s$ is of the form $a_i=g_i/\|g_i\|_2$ and $b_i=0$,
where $g_i$ is an extreme base of $B(f_{\mu_i})$, where $\mu_i\le
\mu$ was the value of $\mu$ at the time when this cutting plane was added.
The lemma will easily follow from the next claim.
\begin{claim}\label{cl:I-s}
The index $1$ is in $I_s$, and  $\left\|\sum_{i\in I_s} t_i a_i\right\|_2=O(n \varepsilon\log(1/\varepsilon))$.
\end{claim}
\begin{proof}
First, we show that
$1\in I_s$. For a contradiction, assume that $1\in I_b$, that is,
$a_1=e_j$ or $a_1=-e_j$ for some $j\in [n]$ and $b_1=-1$.
As noted above, $|b_1+\bar b|=O(n \varepsilon \log(1/\varepsilon))$;
hence, $\bar b>0$ follows (for small enough $\varepsilon$). This is a
contradiction, since $b_i=-1$ for all $i\in I_b$, and $b_i=0$ for all $i\in
I_s$.

Thus, $1\in I_s$, and therefore $b_1=0$. Thus, $|\bar b| =O(n
\varepsilon \log(1/\varepsilon))$. Again, this implies that
$\sum_{i\in I_b} t_i=O(n \varepsilon \log(1/\varepsilon))$. Together
with
$\left\|\sum_{i\in I_b} t_ia_i+\sum_{i\in I_s} t_ia_i\right\|_2=O(\sqrt{n} \varepsilon
\log(1/\varepsilon))$ from (b), we get that $\left\|\sum_{i\in I_s} t_ia_i\right\|_2=O(n\varepsilon
\log(1/\varepsilon))$, as required.
\end{proof}
Let $v=\left(\sum_{i\in I_s} \frac{t_i}{\|g_i\|_2} {g_i}\right)/\left(\sum_{i\in I_s} \frac{t_i}{\|g_i\|_2}\right)$. Since $1\in
I_s$, we have $\sum_{i\in I_s} \frac{t_i}{\|g_i\|_2}\ge
\frac{1}{\Lbt{f}}\ge \frac{1}{\Lb{f}}$. Hence,
it follows that
\[\|v\|_1\le \sqrt{n}\|v\|_2\le \Lb{f}\sqrt n\left\|\sum_{i\in I_s} t_i a_i\right\|_2=O(\Lb{f}
n^{3/2} \varepsilon\log(1/\varepsilon))\le 2\delta\Lb{f}.\]
 Then,
Lemma~\ref{lem:slide-back} is applicable to provide the certificate
for \textsc{Approx-SFM}$(f,\delta)$. Note that the set $W$ with
$f(W)=-\mu$ has been
maintained during the cutting plane algorithm.
\end{proof}

Combining Lemma~\ref{lem:convert} with Theorem~\ref{thm:strongly-2}, and noting that
$\k=O(n)$, we obtain the running time bound $O(n^3 \log^2(n)\EO +\ n^4\log^{O(1)}(n))$.

\paragraph{Comparison to the Lee-Sidford-Wong SFM algorithm}
Let us now compare the above approach to the SFM algorithm described in \cite[Part III]{LSW}. We employ the same cutting plane method, and a common framework is using ring families; our
bucketing argument has been adapted from \cite{LSW}.

Their combinatorial framework is more complex than ours: upper bounds
analogous to the lower bounds $\ell(z)$ are needed, and accordingly, their
algorithm identifies both outgoing and incoming arcs, as well as
removes elements which cannot be contained in any minimizer. The
simple trick
that enables us to work only with lower bounds, and identify only
incoming arcs is repeatedly truncating the value of $f(V)$; thus, we
can bound $\Lb{\lf}$ in terms of $\ell^-(V)$, as in Claim~\ref{cl:L-bound}.

Our black-box approach clearly separates the combinatorial argument
from the cutting plane method, which is used only inside the oracle.
In contrast, these two ingredients cannot be clearly
separated  in \cite{LSW}.
They use the cutting
plane method for the formulation using the Lov\'asz extension and do
not use sliding. Then, they transform the cutting plane
certificate to identify a small norm convex combination in the base
polytope. This is analogous to, but substantially more complicated than our
Lemma~\ref{lem:convert}. In particular, it
is not always possible to identify such a combination, since the
constraints of the feasible region can have large coefficients. In such cases, these large
coefficients can be used to fix some of the variables to 0 and 1, and hence make
progress in terms of the ring family.
In contrast, the certificate from our sliding cutting plane algorithm on $F_\mu$ can be
straightforwardly translated in Lemma~\ref{lem:convert} to satisfy the
requirements of the approximate oracle.


\section{Variants of the geometric rescaling algorithm}\label{sec:remarks}

The framework of Algorithm~\ref{alg:main-weakly} is fairly general, in the sense that both the first-order method used to
generate short convex combinations of normalized vectors of $B(f_\mu)$ and the rescaling used to update the matrix $Q\in\bb{S}_{++}^n$
can be replaced with other alternatives. Here we discuss some of these variants.

\subsection{Replacing von Neumann with Fujishige-Wolfe}
Within Algorithm~\ref{alg:main-weakly}, the role of Algorithm~\ref{alg:Neumann} is to determine a point $y$ with $\|y\|_Q\in O(1/n)$
such that $y$ is a convex combination of points of the form $g/\|g\|_Q$, $g\in B(f_\mu)$. Any algorithm that can produce such output
in time polynomial in $n$ can be used in place of von Neumann algorithm. In particular, Fujishige-Wolfe
can be adapted to the rescaling setting of Algorithm~\ref{alg:main-weakly}.

As for von Neumann's algorithm, the only modifications that are
required are the following. First of all, given a matrix $Q\in\bb{S}_{++}^n$, we use the $Q$-norm and $Q$-scalar product to normalize
the elements of $B(f_\mu)$. We maintain a set $X$ of affinely independent elements of the form  $g/\|g\|_Q$, $g\in B(f_\mu)$. At every major cycle --
that is, when $X$ is a corral -- the algorithm computes the projection $y$ of the origin to the affine hull of $X$, which belongs
to the relative interior of $\conv(X)$ because $X$ is a corral. At any major cycle we are only interested in
knowing whether or not $Qy\in \intr(\Sigma_\mu)$. Therefore, we call  \textsc{GreedyMin}$(f_\mu,Qy)$, obtaining a minimizer $g\in B(f_\mu)$ as output.
 If $\scalar{g,y}_Q\leq 0$,
then we set $X':= X\cup \{g/\|g\|_Q\}$.
If $\scalar{g,y}_Q> 0$ (that is, $Qy\in \intr(\Sigma_\mu)$), then we determine $W\subset V$ such that $f(W)<-\mu$ and slide $f$ by setting $\mu:=-f(W)$,
just as in line \ref{li:vN-sliding} of the
sliding-von Neumann algorithm (Algorithm~\ref{alg:Neumann}).
We update $g$ as in line \ref{li:new g},
 and then set $X'=X\cup\{g/\|g\|_Q\}$.  In both cases, we proceed with either another major cycle, if $X'$ is still a corral,
or a minor cycle if $X'$ is not a corral. Note, in particular, that in both cases we only increase the set $X$ when $\scalar{g,y}_Q\leq 0$.
The next statement follows immediately from \cite[Theorem 4]{chakrabarty14} and from the same arguments used in the proof of Lemma~\ref{lemma:sliding VN}.
\begin{theorem}
Given a value $\mu\ge\max\{0,-f(V)\}$, a set $W\subseteq V$ with $f(W)=-\mu$, a matrix $Q\in \bb{S}^{n}_{++}$, and an $\varepsilon>0$,
within $O(1/\varepsilon^2)$ iterations  (major and minor cycles), Wolfe's algorithm computes a value $\mu'\ge \mu$ and a set $W'\subseteq V$ with $f(W')=-\mu'$,
bases $g_1,\ldots,g_k\in B(f_{\mu'})$,  $x\in \R^k$, $y\in \R^n$  such that $k\leq n$,
$y=\sum_{i=1}^k{x_i g_i/\|g_i\|_Q}$, $\vec e^\T x=1$, $x\ge 0$, and  $\|y\|_Q\le \varepsilon$.
\end{theorem}

In particular, within  Algorithm~\ref{alg:main-weakly}, we need to fix $\varepsilon=1/(20n)$, which implies that the number of iterations
required for each call of Fujishige-Wolfe is $O(n^2)$.  Each iteration requires at most one call to $\textsc{GreedyMin}$, needing $O(n)$ oracle calls,
plus $O(n^2)$ arithmetic operations (the iteration complexity is dominated by the computation of the projection of the origin onto the affine hull of $X$).
Given that the number of rescalings needed is $O(n\log(n\Lbt{f}))$, the total running time of Algorithm~\ref{alg:main-weakly} where the sliding-von Neumann algorithm is
replaced with Fujishige-Wolfe is $O((n^4 \cdot \EO + n^5)\log (n\Lbt{f}))$.


\subsection{Rank-1 rescalings}

The multi-rank rescaling \eqref{eq:rescale} used in Algorithm~\ref{alg:main-weakly} can be replaced by a rank-1 update. Here we discuss two possible such updates.
In both cases, the analysis relies on the following lemma, which is an immediate consequence of \cite[Lemmas 4 and 5]{rothvoss}, when applied within
the setting of submodular function minimization.

\begin{lemma}[{Hoberg and Rothvo\ss{}~\cite{rothvoss}}]\label{lemma:HR}
Let $\mu\ge\max\{0,-f(V)\}$ and  $Q\in \bb{S}^{n}_{++}$ such that, for $R=Q^{-1}$, $F_\mu\subseteq E(R)$.
Let $y=\sum_{i=1}^{k} x_i g_i/\|g_i\|_Q$ where $g_1,\ldots,g_k\in B(f_\mu)$,
$x\in\R^k_+$, $\vec{e}^\T x=1$, and assume that there is a set $I\subset [k]$
such that the vector $z:=\sum_{i\in I} x_i \frac{g_i}{\|g_i\|_Q}$ satisfies
\begin{equation}\label{eq:HR-condition}
\frac{\|y\|_Q}{\|z\|_Q}\leq \frac{1}{3\sqrt{n}}.
\end{equation}
Define
\begin{equation}\label{eq:rank-1}
R':=\frac{1}{(1+27 n)}\left(R+3\frac{ zz^\T }{\|z\|_Q^2}\right).
\end{equation}
Then $F_\mu\subseteq E(R')$ and $\det(R')\geq (9/4) \det(R)$.
\end{lemma}

Next we show two ways of using \Call{Sliding von Neumann}{} to produce vectors $y$ and $z$ as in the theorem above. In particular, the rescaling defined by~\eqref{eq:rank-1}
can be used within  Algorithm~\ref{alg:main-weakly} in place of the multi-rank rescaling \eqref{eq:rescale}. Note that the analysis of the algorithm is then identical
to that of Algorithm~\ref{alg:main-weakly} (simply replace Lemmas~\ref{lem:F-in} and~\ref{lem:dual-vol-dec} by Lemma~\ref{lemma:HR} in the analysis),
hence we can determine a minimizer for $f$ within $O(n\log(n\Lbt{f}))$ rescalings of the form \eqref{eq:rank-1}.

\paragraph{Betke's rescaling}

Betke~\cite{betke} proposed the following rescaling. Within Algorithm~\ref{alg:main-weakly}, fix $\varepsilon=1/(3(n+1)\sqrt{n})$. Assume that
the vector $y=\sum_{i=1}^{k} x_i g_i/\|g_i\|_Q$ returned by each call of \Call{Sliding von Neumann}{$f,\mu,W,Q,\varepsilon$} is expressed as a convex combination of at most
$n+1$ terms (this is not instantly guaranteed by the sliding von-Neumann algorithm, but it can be done a-posteriori by Carath\'eodory's theorem).
Since $\sum_{i=1}^{k} x_i=1$, $x_i\geq 0$ for all $i\in[k]$, it follows that there exists $h\in [k]$ such that $x_h\geq 1/k\geq 1/(n+1)$.  The set
$I:=\{h\}$ satisfies condition \eqref{eq:HR-condition} in Lemma~\ref{lemma:HR}; indeed, in this case $z=x_h \frac{g_h}{\|g_h\|_Q}$, and we have
$$\frac{\|y\|_Q}{\|z\|_Q}=\frac{\|y\|_Q}{x_h}\leq \frac{1/(3(n+1)\sqrt{n})}{1/(n+1)}= \frac{1}{3\sqrt{n}}.$$

Observe that, with Betke's rescaling, the sliding-von Neumann algorithm requires $\ceil{\varepsilon^{-2}}=O(n^3)$ iterations to produce such a vector. Recall that each iteration
of von-Neumann algorithm requires time $O(n\cdot E+n^2)$, for a total running time of $O(n^4\cdot EO+n^5)$ for each call to sliding-von Neumann. The vector $y$ returned is expressed as a convex combination of $O(n^3)$ vectors of the form $g/\|g\|_Q$, $g\in B(f_\mu)$, hence we need $O(n^5)$ operation to express the vector $y$ as a convex combination of at most $n+1$ such terms. Finally,
the total number of rescalings is $O(n\log(n\Lbt{f}))$. This means that the variant of Algorithm~\ref{alg:main-weakly} using Betke's rescaling instead of \eqref{eq:rescale} requires time $O((n^5\cdot EO+n^6)\log(n\Lbt{f}))$,
as opposed to the $O((n^4\cdot EO+n^5)\log(n\Lbt{f}))$ ensured by the multi-rank update \eqref{eq:rescale}.

\paragraph{Hoberg and Rothvo\ss{} rescaling}

Hoberg and Rothvo\ss{} \cite{rothvoss} provide a randomized selection rule for the vector $z$ in the rank-1 rescaling~\eqref{eq:rank-1}. The rule is based on the following.
\begin{lemma}[{Hoberg and Rothvo\ss{} \cite[Section 2.1]{rothvoss}}]\label{lem:HR-random}
Let $v_1,\ldots,v_k\in\R^n$ be such that $\sum_{i=1}^k \|v_i\|_2=1$. Let $u\in\R^n$, $\|u\|_2=1$, be chosen uniformly at random, and let $I=\{i\in [k]\st u^\T v_i\geq 0\}$.
 Then with constant
probability
\begin{equation}\label{eq:HR-probabilistic}
\left\|{\sum_{i\in I} v_i}\right\|_2\geq  \frac{1}{4\sqrt{\pi n}}.
\end{equation}
\end{lemma}

Within Algorithm~\ref{alg:main-weakly}, fix $\varepsilon\leq 1/(12\sqrt{\pi}n)$. Given the vector $y=\sum_{i=1}^{k} x_i g_i/\|g_i\|_Q$ returned by a call of \Call{Sliding von Neumann}{$f,\mu,W,Q,\varepsilon$}, we  let $v_i=x_i Q^{1/2}g_i/\|g_i\|_Q$ for $i=1,\ldots, k$;  by construction, $\sum_{i=1}^k \|v_i\|_2=1$.
We can randomly sample the set $I$ as in Lemma~\ref{lem:HR-random} that satisfies the requirements of  Lemma~\ref{lemma:HR}; we repeat to sample $I$ until condition~\eqref{eq:HR-probabilistic} is satisfied. At this point,
the vector $z=\sum_{i\in I }x_i \frac{g_i}{\|g_i\|_Q}$ satisfies condition \eqref{eq:HR-condition}, because
$$\frac{\|y\|_Q}{\|z\|_Q}\leq 4\sqrt{\pi n}\cdot \varepsilon\leq \frac{1}{3\sqrt{n}}.$$
Observe that each call to \Call{Sliding von Neumann}{$f,\mu,W,Q,\varepsilon$} requires $\ceil{\varepsilon^{-2}}=O(n^2)$ iterations, for a total time of $O(n^3\cdot EO+n^4)$ per call, just as for Algorithm~\ref{alg:main-weakly}. Checking if the random set $I$ satisfies
condition~\eqref{eq:HR-probabilistic} takes time $O(mk)$, where $k$ is the number of vectors in the convex combination defining $y$, which is bounded by the number $O(n^2)$ of iterations of von Neumann. The expected number of times we need to generate a random set $I$ before condition~\eqref{eq:HR-probabilistic} is verified is constant, hence in expected $O(n^3)$ arithmetic operation we can compute the rescaling direction $z$ after each call \Call{Sliding von Neumann}{$f,\mu,W,Q,\varepsilon$}. Hence the overall expected running time of the algorithm is   $O((n^4\cdot EO+n^5)\log(n\Lbt{f}))$, just as for Algorithm~\ref{alg:main-weakly}.

\subsection{Rank-1 pullback}
The pullback framework described in Section~\ref{sec:farkas} can also be adapted to the use of rank-1 updates of the form~\eqref{eq:rank-1}. All arguments in Section~\ref{sec:farkas} proceed without any modification, provided that we replace Lemma~\ref{lemma:pullback} (used for the multi-rank rescaling~\eqref{eq:rescale}) with the following analogous statement for the rank-1 rescaling \eqref{eq:rank-1}.
\begin{lemma}
\label{lemma:pullback-rank 1} Let $A\in\R^{n\times p}$, $R\in\bb{S}_{++}^n$, and $Q=R^{-1}$. Let $x\in\R^p_+$ and $I\subseteq [p]$ such that the vectors $y:=\sum_{i=1}^p x_i \frac{a_i}{\|a_i\|_Q}$ and $z:=\sum_{i\in I} x_i \frac{a_i}{\|a_i\|_Q}$ satisfy
$$\frac{\|y\|_Q}{\|z\|_Q}\leq \gamma.$$
Define
$$R':=\frac{1}{(1+3\gamma^2)}\left(R+3\frac{ zz^\T }{\|z\|_Q^2}\right),$$
and $Q'=(R')^{-1}$. For every $v\in \R^n$, there exists $\mu\in\R^p_+$ such that $\|v+A\mu\|_Q\leq \|v\|_{Q'}$.
\end{lemma}
\begin{proof}
First note that $$Q':=(1+3\gamma^2)\left(Q-\frac{3}{4} \frac{ Qzz^\T Q^\T }{\|z\|_Q^2}\right).$$
Let $v\in \R^n$, and define $$u=\frac{RQ'v}{1+3\gamma^2},$$
Note that $$v=R'QRQ'v=u+3\frac{\scalar{z,u}_Q}{\|z\|_Q^2}z.$$
Thus, if we define
$$\mu_i=\left\{\begin{array}{ll}
                 3\frac{x_i}{\|a_i\|_Q} \frac{|\scalar{z,u}_Q|}{\|z\|_Q^2}& i\notin I \\
                 \max\left\{0,-6\frac{x_i}{\|a_i\|_Q} \frac{\scalar{z,u}_Q}{\|z\|_Q^2}\right\} & i\in I
               \end{array}
\right.$$
we obtain $$v+A\mu=u+ 3\frac{|\scalar{z,u}_Q|}{\|z\|_Q^2}y.$$

Note that  $\mu\geq 0$ and
\begin{eqnarray*}
\|v\|^2_{Q'}&=&v^TQ'R(QR'Q)RQ'v=(1+3\gamma^2)u^\T\left(Q+3\frac{Qzz^\T Q}{\|z\|_Q^2}\right)u\\
&=&\|u\|_Q^2 (1+3\gamma^2)(1+3\alpha^2),
\end{eqnarray*}
where we define $\alpha:= \frac{\scalar{z,u}_Q}{\|z\|_Q\|u\|_Q}$.
On the other hand
$$\|v+A\mu\|_Q\leq \|u\|_Q+3\frac{|\scalar{z,u}_Q|}{\|z\|_Q^2}\|y\|_Q=\|u\|_Q(1+3|\alpha| \frac{\|y\|_Q}{\|z\|_Q})\leq \|u\|_Q(1+3|\alpha|\gamma).$$
The statement now follows by noting that $(1+3\gamma^2)(1+3\alpha^2)-(1+3|\alpha|\gamma)^2=3(|\alpha|-\gamma)^2\geq 0$.
\end{proof}


\section*{Acknowledgments.}
We would like to thank the anonymous referees for the many comments and suggestions that have helped improve the presentation of the paper.

\appendix

\section{Proofs of technical lemmas}

\subsection{Proofs of Section~\ref{sec:prelim}}$\,$

\begin{proof}[Proof of Lemma~\ref{thm:approx-convert}] Assume w.l.o.g that $z(v_1)\le z(v_2)\le \ldots\le z_(v_n)$, and denote $S_i=\{v_1,\ldots,v_i\}$. The hypothesis of the statement can
be written as $-\min_{x\in B(f)} z^\top (x-z) \le \varepsilon$. If we let $f'\st 2^V\mapsto \R$ be the submodular function defined by $f'(S)=f(S)-z(S)$, observing that $x\in B(f)$ if and only if $x-z\in B(f')$, the previous expression implies
$-\min_{y\in B(f')}z^\top y\le \varepsilon$. By \eqref{eq:vertex-cost}, and observing that $f'(V)=0$, this can be written as

\begin{eqnarray*}
\varepsilon &\ge& \sum_{i=1}^{n-1}(f(S_i)-z(S_i))(z(v_{i+1})-z(v_i))\\
&=& \int_{-\infty}^{+\infty} f(\{v\st z(v)\le t\})-z(\{v\st z(v)\le t\})\, dt\\
&\ge& \int_{-\eta}^{+\eta} f(\{v\st z(v)\le t\})-z(\{v\st z(v)\le t\})\, dt,\\
\end{eqnarray*}
where the last inequality holds trivially for every $\eta>0$ since $f(S)-z(S)\ge 0$ for all $S\subseteq V$ because $z\in B(f)$, and where the equality holds because

\[
\{v\st z(v)\le t\}=
\begin{cases}
\emptyset & \mbox{if } t<z(v_1)\\
S_i &\mbox{if } z(v_i)\leq t < z(v_{i+1})\\
V & \mbox{if } t\ge z(v_n).
\end{cases}
\]

It follows from the above that, for every $\eta>0$, there exists $t\in[-\eta,\eta]$ such that $f(\{v\st z(v)\le t\})-z(\{v\st z(v)\le t\})\le \varepsilon/2\eta$. Setting $S:=\{v\st z(v)\le t\}$, this implies
\[f(S)-z^-(V)\le z^+(S)-z^-(V\sm S)+ \varepsilon/2\eta\le n|t| +\varepsilon/2\eta\le n\eta +\varepsilon/2\eta.\]
Choosing $\eta=\sqrt{\varepsilon/2n}$ gives $f(S)-z^-(V)\le  \sqrt{2n\varepsilon}$, as required.
\end{proof}

\subsection{Proofs of Section~\ref{sec:weakly-analysis}}
\label{appendix}
\begin{lemma}
  \label{lemma:E(R) contained at rescaling}
The property that  $F_\mu\subseteq E(R)$ is preserved whenever Algorithm~\textsc{Rescaling-SFM} performs a rescaling.
\end{lemma}
\begin{proof}
Assume that at a given step of the algorithm  $F_\mu\subseteq E(R)$ holds, and we rescale $R$ to $R'$, where
\[
R'=\frac{1}{(1+\varepsilon)^2}\left(R+\sum_{i=1}^k \frac{x_i }{\|g_i\|_Q^2}g_i g_i^\T\right)
\]
for the vector $x$ returned by the sliding von Neumann algorithm.

We show $F_\mu\subseteq E(R')$.  Consider an arbitrary point $z\in F_\mu$;
then $ g_i^\T z\ge 0$ for all $i\in[k]$ and, by the induction hypothesis, $\|z\|^2_{R}\le 1$ because $z\in E(R)$.

Recall that, in the algorithm, the vector $y=\sum_{i=1}^k x_i \frac{g_i}{\|g_i\|_Q}$ satisfies $\|y\|_Q\leq\varepsilon$.
By the Cauchy-Schwartz inequality, we have $y^\T z=y^\T Q^{1/2}{Q^{-1/2}}z\leq \|y\|_Q\|z\|_R\leq \varepsilon$, and similarly $ g_i^\T z\le \|g_i\|_Q\|z\|_R\leq \|g_i\|_Q$ for every $i\in[k]$. We then have
\begin{eqnarray*}
\|z\|^2_{R'}&=&\frac{1}{(1+\varepsilon)^2} z^\T \left(R+\sum_{i=1}^k\frac{x_i}{\|g_i\|_Q^2} g_ig_i^\T\right) z
=\frac{1}{(1+\varepsilon)^2}\left(\|z\|^2_{R}+\sum_{i=1}^k x_i \left(\frac{g_i^\T z}{\|g_i\|_Q} \right)^2\right)\\
&\le&\frac{1}{(1+\varepsilon)^2}\left(1+\sum_{i=1}^k x_i \frac{g_i^\T z}{\|g_i\|_Q} \right) = \frac{1+ y^\T z}{(1+\varepsilon)^2}\leq 1,
\end{eqnarray*}
where the first inequality follows from the facts that $\|z\|_R\leq 1$, $x\geq 0$, and $0\le g_i^\T z\le \|g_i\|_Q$ for all $i\in[k]$, while the second follows from  $y^\T z\leq \varepsilon$. Consequently, $z\in E(R')$, completing the proof.
\end{proof}

\begin{proof}[Proof of Lemma~\ref{lem:dual-vol-dec}]
Let $R$ and $R'$ denote the matrix before and after the rescaling. Let $X=\sum_{i=1}^k
  x_ig_ig_i^\T/\|g_i\|_Q^2$; hence $R'={(R+X)}/(1+\varepsilon)^2$. The ratio of the
  two determinants is
\[
\frac{\det(R')}{\det(R)}=\frac{\det(R+X)}{(1+\varepsilon)^{2n}\det(R)}=\frac{\det\left(I_n+R^{-1/2}XR^{-1/2}\right)}{(1+\varepsilon)^{2n}}
\]
Now $R^{-1/2}=Q^{1/2}$, and $Q^{1/2}XQ^{1/2}$ is a positive
semidefinite matrix. From the above, the inequality $\det(I_n+M)\ge 1+\tr(M)$ for every $M\in\bb{S}^n_{++}$, and the linearity of the trace, we get
\[
\frac{\det(R')}{\det(R)}\ge \frac{1+\tr(Q^{1/2}XQ^{1/2})}{(1+\varepsilon)^{2n}}=\frac{1}{(1+\varepsilon)^{2n}}
\left(1+\sum_{i=1}^k \frac{x_i}{\|g_i\|_Q^2} \tr(Q^{1/2}g_ig_i^\T Q^{1/2})\right).
\]
Finally,
$\tr(Q^{1/2}{g_ig_i^\T }Q^{1/2})=\tr(g_i^\T Qg_i)=\|g_i\|_Q^2$. Therefore
we conclude
\[
\frac{\det(R')}{\det(R)}\ge \frac{1+\sum_{i=1}^k x_i}{(1+\varepsilon)^{2n}}=\frac{2}{(1+\varepsilon)^{2n}}\ge \frac{16}{9},
\]
where the last inequality follows from $\varepsilon=\frac{1}{20n}$.
 \end{proof}

\begin{proof}[Proof of Lemma~\ref{lem:a-k-Q-bound}]
Let $\bar g=\arg\min_{g\in \Gamma}\|g\|_Q/\|g\|_2$.
Then, by \eqref{eq:R-form},
\begin{align}\label{eq:g bound}
\frac{\|\bar g\|^2_Q}{\|\bar g\|^2_2}\sum_{g\in\Gamma} \gamma_g &\leq \sum_{g\in\Gamma}\gamma_g \frac{\| g\|^2_Q}{\|g\|^2_2}=\sum_{g\in\Gamma} \gamma_g
\frac{g^\T Q g}{\|g\|^2_2}
= \tr\left(Q \sum_{g\in\Gamma} \gamma_g
   \frac{g g^\T}{\|g\|^2_2}\right)\\
&= \tr(Q(R-\alpha I_n))=\tr(I_n-\alpha Q)=n-\alpha \tr(Q)<n~. \nonumber
\end{align}
The final inequality holds since $Q$ is positive definite.

Note from \eqref{eq:R-form} that $\tr(R)=\alpha n+\sum_{g\in\Gamma} \gamma_g\le n+\sum_{g\in\Gamma} \gamma_g$. The latter and the inequality $\det(M)^{1/n}\le \tr(M)/n$ for every $M\in\bb{S}^n_{++}$ imply that $\sum_{g\in\Gamma}\gamma_g\geq \tr(R)-n\ge n(\det(R)^{1/n}-1)$. The statement now follows from \eqref{eq:g bound}.
 \end{proof}

\bibliographystyle{abbrv}
\bibliography{rescaled}

\end{document}